\documentclass[reqno]{tac}
\usepackage{latexsym,amssymb,bbm,amsmath,mathabx}
\usepackage[pdftex]{graphicx}
\usepackage[scr=boondox]{mathalfa}
\usepackage[mathscr]{euscript}
\usepackage{rotating}
\usepackage{color}
\usepackage[color,emtex,matrix,arrow]{xy}
\usepackage{xypic}
\usepackage{comment}
\usepackage{bbold}
\usepackage{calligra}
\usepackage[T1]{fontenc}
\usepackage{calrsfs}
\usepackage{tikz-cd}

\title[Alternative description of symmetric monoidal categories]{An alternative description of symmetric monoidal categories, and symmetric 2-groups}
\author{Josep Elgueta}
\date{} 
\address{Departament de Matem\`atiques, Universitat Polit\`ecnica de Catalunya}
\eaddress{josep.elgueta@upc.edu}

\numberwithin{equation}{section}

\newcommand{\AAA}{\mathbb{A}}

\newcommand{\CC}{\mathbb{C}}
\newcommand{\DD}{\mathbb{D}}
\newcommand{\EE}{\mathbb{E}}
\newcommand{\FF}{\mathbb{F}}
\newcommand{\GG}{\mathbb{G}}

\newcommand{\RR}{\mathbb{R}}

\newcommand{\ZZ}{\mathbb{Z}}

\newcommand{\Aa}{\mathcal{A}}
\newcommand{\Cc}{\mathcal{C}}
\newcommand{\Dd}{\mathcal{D}}

\newcommand{\Asf}{\mathsf{A}}
\newcommand{\Bsf}{\mathsf{B}}

\newcommand{\Gsf}{\mathsf{G}}
\newcommand{\Hsf}{\mathsf{H}}

\newcommand{\Msf}{\mathsf{M}}

\newcommand{\Rsf}{\mathsf{R}}

\newcommand{\csf}{\mathsf{c}}

\newcommand{\hsf}{\mathsf{h}}

\newcommand{\zsf}{\mathsf{z}}

\begin{document}

\maketitle

\begin{abstract}
An equivalent description of a symmetric monoidal category is introduced in which, instead of separate associator and commutator isomorphisms satisfying the usual coherence axioms, we simply have associo-commutator isomorphisms satisfying their own coherence laws. In particular, this yields an alternative description of a symmetric 2-group and leads to a cohomological classification of these objects in terms of Eilenberg-MacLane cubical cohomology for abelian groups.
\end{abstract}

\section{Introduction}

The notion of symmetric monoidal category dates back to the 1960s \cite{MacLane-1963}. It is by now a well established structure used across many areas of mathematics, but also in logic, theoretical computer science and theoretical physics. The notion naturally arises as the categorification of the notion of commutative monoid. Thus instead of a set $M$, a set theoretical map $\cdot:M\times M\to M$ and a distinguished element $1$ satisfying the usual associative, commutative and unital laws  we have a category $\Cc$, a functor $\cdot:\Cc\times\Cc\to\Cc$ and a distinguished object $1$ together with suitable natural associativity, commutativity and left and right unit isomorphisms satisfying their own coherence laws (see \cite{MacLane-1998} for the precise definition).

As a matter of fact, however, a set $M$ equipped with a binary operation $\cdot:M\times M\to M$ and an identity element $1\in M$ with respect to $\cdot$ is a commutative monoid if and only if the following {\em associo-commutative property} holds for every $x,y,z,t\in M$ (for short, the symbol $\cdot$ is omitted):
\begin{equation}\label{associocommutativitat}
(xy)(zt)=(xz)(yt)
\end{equation}
Indeed, if $\cdot$ is associative and commutative we have
\[
(xy)(zt)=((xy)z)t=(x(yz))t=(x(zy))t=((xz)y)t=(xz)(yt).
\]
Conversely, let $\cdot$ be such that (\ref{associocommutativitat}) holds. Since $1$ acts as an identity we have
\[
x(yz)=(x1)(yz)=(xy)(1z)=(xy)z
\]
and 
\[
xy=(1x)(y1)=(1y)(x1)=yx.
\]
This suggests that there should exist an alternative definition of a symmetric monoidal category where, instead of separate associator and commutator isomorphisms satisfying the usual coherence axioms, we simply have {\em associo-commutator} isomorphisms
\[
b(x,y,z,t):(xy)(zt)\overset{\cong}{\longrightarrow}(xz)(yt)
\]
satisfying their own coherence conditions. 

The purpose of this work is to make this structure precise, together with the corresponding notions of morphism and 2-morphism, and to prove that the 2-category so obtained is indeed equivalent (in fact, isomorphic) to the 2-category of symmetric monoidal categories. As a consequence, we obtain alternative descriptions of the derived notions of internal (co)commutative (co)monoid, symmetric 2-group, PROP, operad, dagger symmetric monoidal category, etc.

This work was originally motivated by the corresponding alternative description of symmetric 2-groups (that is, symmetric monoidal groupoids in which every object is invertible up to isomorphism) and by the cohomological classification of these objects up to equivalence. I call the resulting structures {\em AC (Associo-Commutative) 2-groups}. More broadly, my interest lies in categorical analogues of rings and modules over a ring (particularly 2-vector spaces over a field) and in their corresponding cohomological classifications. The notion of categorical ring, or {\em 2-ring}, was first introduced by Quang \cite{Quang-1987}, who called them {\em Ann-categories}. The point is that, for the purpose of classifying 2-rings up to equivalence, it is advantageous to regard the underlying symmetric 2-group as an AC-2-group. Indeed, Jibladze and Pirashvili \cite{Jibladze-Pirashvili-2007} proved that 2-rings are classified by MacLane's cohomology of rings, introduced by this author in the 1950s to classify the extensions of rings \cite{MacLane-1958b}. More precisely, they proved that, up to equivalence, a 2-ring $\RR$ is completely specified by a triple $(\Rsf,\Msf,\zsf)$ consisting of a ring $\Rsf$, an $\Rsf$-bimodule $\Msf$ and a MacLane 3-cocycle of $\Rsf$ with coefficients in $\Msf$. However, unlike 3-cocycles in classical group cohomology, which are simply given by a map $\zsf:R^3\to M$ satisfying the usual 3-cocycle condition, a MacLane 3-cocycle $\zsf$ consists of a quadruple of maps $(\zsf_+,\zsf_\cdot,\zsf_{\cdot +},\zsf_{+\cdot})$, with $\zsf_+:R^4\to M$ and $\zsf_\cdot,\zsf_{\cdot +},\zsf_{+\cdot}:R^3\to M$, subject to a collection of equations relating them, together with certain normalizations conditions (see \cite{Jibladze-Pirashvili-2007}, \cite{Quang-2013}). In particular, the map $\zsf_+$ turns out to be a (normalized) 3-cocycle in the Eilenberg-MacLane cubical complex \cite{Eilenberg-MacLane-1950-I}, and this is precisely the invariant that classifies the underlying symmetric 2-group of $\RR$ when thought of as an AC-2-group. It seems likely that identifying the right cohomology classifying 2-modules over a given 2-ring will be easier in terms of this invariant than in terms of the pair $(\hsf,\csf)$ providing the usual cocycle description, up to equivalence, of a symmetric (or braided) 2-group (\cite{Sinh-1975}, \cite{Joyal-Street-1993}); see \S~\ref{comentari_final}. 

The paper is organized as follows. In Section 2, we introduce the notion of an AC-category, together with the corresponding notions of AC-functor and AC-natural transformation, and prove that the resulting 2-category is equivalent (in fact, isomorphic) to the 2-category of symmetric monoidal categories (cf. Theorem~\ref{teorema_equivalencia}). As a consequence, we obtain the corresponding strictification theorem for AC-categories, according to which every AC-category is equivalent to a unital one whose associo-commutator $b(x,y,z,t)$ is the identity whenever either of the middle arguments $y$ or $z$ is the unit object. We also make the relationship between these AC-categories and the symmetric strict monoidal categories precise (cf. Proposition~\ref{AC-categories_semistrictes}). In Section 3, we recall the definition of the (normalized) cubical complex of an abelian group, due to Eilenberg and Mac Lane, and the associated cubical cohomology of abelian groups. Finally, in Section 4, we show that AC-2-groups are classified by this cubical cohomology (cf. Theorem~\ref{classificacio_AC-2-grups}), providing an alternative cohomological description of symmetric 2-groups. 

\subsection*{Related work}
The alternative description discussed in this work is valid only for symmetric monoidal categories. In fact, the basic coherence law required of the associo-commutators automatically implies the symmetry condition (cf. Proposition~\ref{cond_simetria} below). For an alternative description of arbitrary braided monoidal categories, the reader is referred to \cite{Davydov-Runkel-2015}. In this work, the authors introduce the notion of (right) \emph{unital $b$-category} and show that it is equivalent to the classical notion of braided monoidal category (\cite{Davydov-Runkel-2015}, Theorem~2.4). Still another alternative description of braided monoidal categories can be found in \cite{Aguiar-Mahajan-2010}. In Chapter~6, the authors introduce the notion of \emph{2-monoidal category}. Roughly speaking, it is a category $\mathcal C$ simultaneously equipped with two monoidal structures $(\cdot,1,a,l,r)$ and $(\cdot',1',a',l',r')$ along with natural morphisms
\begin{align*}
\zeta_{x,y,z,t}&:(x\cdot y)\cdot'(z\cdot t)\to(x\cdot' z)\cdot(y\cdot' t)
\\
\Delta_1&:1\to 1\cdot' 1
\\
\mu_{1'}&:1'\cdot 1'\to 1'
\\
\epsilon&:1\to 1'
\end{align*}
satisfying the appropriate coherence conditions. It is called \emph{strong} when these morphisms are invertible. In Propositions~6.10-6.11 it is shown that a braided monoidal category is the same as a \emph{strong} 2-monoidal category. In fact, the authors also prove (Proposition~6.13) that symmetric monoidal categories are the same as strong \emph{braided} 2-monoidal categories --that is, strong 2-monoidal categories whose monoidal structures are actually braided. However, the notion of AC-category introduced below looks like a significantly cleaner alternative presentation of a symmetric monoidal category.

Finally, let us mention the work by Došen and Petrić \cite{Dorsen-Petric-2007}, where they also provide an alternative description of symmetric monoidal categories in terms of isomorphisms of the form $(xy)(zt)\cong(xz)(yt)$, which they term \emph{medial commutativities}. Their approach differs substantially from ours and is somewhat more intricate to follow.

\subsection*{Acknowledgements}

I thank the anonymous referee for their careful reading, for helpful suggestions that improved the presentation, and for pointing out an error in the proof of Proposition~\ref{AC-categories_semistrictes}. I am also grateful to Nathanael Arkor for drawing my attention to the work of Došen and Petrić as well as the notion of 2-monoidal categories. Remark~\ref{remarca_final} was likewise motivated by his comments.

\section{An alternative description of a symmetric monoidal category}

We assume the reader is already familiar with the notions of symmetric monoidal category, symmetric monoidal functor and monoidal natural transformation. Some helpful references are \cite{MacLane-1998}, \cite{Aguiar-Mahajan-2010}, \cite{Yau-2024-I}.

\begin{definition}\label{AC-categoria}
An {\em  AC-category} is a sixtuple $\CC=(\Cc,\cdot,1,b,l,r)$ consisting of:
\begin{itemize}
\item a category $\Cc$;
\item a functor $\cdot:\Cc\times\Cc\to\Cc$ (the {\em product functor});
\item a distinguished object $1$ in $\Cc$ (the {\em unit object});
\item natural isomorphisms $b(x,y,z,t):(xy)(zt)\to(xz)(yt)$ for every objects $x,y,z,t$ in $\Cc$ (the {\em associo-commutators});
\item natural isomorphisms $l_x:1x\to x$ for every object $x$ in $\Cc$ (the {\em left unitors});
\item natural isomorphisms $r_x:x1\to x$ for every object $x$ in $\Cc$ (the {\em right unitors}).
\end{itemize}
Moreover, these data must satisfy the following axioms (for short, we omit the symbol $\cdot$ between objects):
\begin{itemize}
\item[{\rm (acc1)}] {\em ($4\times 4$ axiom)} for every objects $x,y,z,t,x',y',z',t'$ in $\Cc$ the following diagram commutes:
\[
\xymatrix@C=4pc{
((xy)(zt))((x'y')(z't'))\ar[rr]^-{b(xy,zt,x'y',z't')}\ar[d]_-{b(x,y,z,t)\cdot b(x',y',z',t')}
&&
((xy)(x'y'))((zt)(z't'))\ar[d]^-{b(x,y,x',y')\cdot b(z,t,z',t')} 
\\
((xz)(yt))((x'z')(y't'))\ar[d]_-{b(xz,yt,x'z',y't')} 
&&
((xx')(yy'))((zz')(tt'))\ar[d]^{b(xx',yy',zz',tt')}
\\
((xz)(x'z'))((yt)(y't'))\ar[rr]_-{b(x,z,x',z')\cdot b(y,t,y',t')} && ((xx')(zz'))((yy')(tt'))
}
\] 

\item[{\rm (acc2)}] {\em (unital axioms)} for every objects $x,y$ in $\Cc$ the following diagrams commute:
\begin{equation}\label{eq00}
\xymatrix{
(x1)(y1)\ar[rr]^-{b(x,1,y,1)}\ar[d]_-{r_x\cdot r_y} && (xy)(11)\ar[d]^-{id\cdot l_1}
\\
xy && (xy)1\ar[ll]^-{r_{xy}} }
\quad 
\xymatrix{
(1x)(1y)\ar[rr]^-{b(1,x,1,y)}\ar[d]_-{l_x\cdot l_y} && (11)(xy)\ar[d]^-{r_1\cdot id}
\\
xy && 1(xy)\ar[ll]^-{l_{xy}} }
\end{equation}
\begin{equation}\label{eq0}
\xymatrix{
(xy)(11)\ar[rr]^-{b(x,y,1,1)}\ar[d]_-{id\cdot l_1} && (x1)(y1)\ar[d]^-{r_x\cdot r_y}
\\
(xy)1\ar[rr]_-{r_{xy}} && xy }
\quad 
\xymatrix{
(11)(xy)\ar[rr]^-{b(1,1,x,y)}\ar[d]_-{r_1\cdot id} && (1x)(1y)\ar[d]^-{l_x\cdot l_y} 
\\
1(xy)\ar[rr]_-{l_{xy}} && xy }
\end{equation} 
\item[{\rm (acc3)}] {\em (normalization axiom)} for every objects $x,y$ in $\Cc$ we have
\begin{equation}\label{cond_simetria_b}
b(x,1,1,y)=id_{(x1)(1y)}\,;
\end{equation}
\end{itemize}
An AC-category is called {\em unital} when all left and right unitors $l_x,r_x$ are identities, {\em semistrict} when it is unital and the associo-commutator $b(x,y,z,t)$ is the identity whenever either of the middle arguments $y$ or $z$ is the unit object, and {\em strict} when it is unital and $b(x,y,z,t)$ is the identity for every quadruple $(x,y,z,t)$.
\end{definition}

It turns out that the associo-commutators in every AC-category automatically satisfy the following {\em symmetry condition}.

\begin{proposition}\label{cond_simetria}
For every objects $x,y,z,t$ in $\Cc$ we have
\begin{equation}
b(x,z,y,t)\,b(x,y,z,t)=id_{(xy)(zt)}.
\end{equation} 
\end{proposition}

\begin{proof}
Let $S(x,y,z,,t)$ be the natural isomorphism given by the composite
\[
S(x,y,z,t)=b(x,z,y,t)\,b(x,y,z,t).
\]
We have to see that $S(x,y,z,t)$ is the identity for each $(x,y,z,t)$. For short, let us write $u,u'$ instead of $(x,y,z,t),(x',y',z',t')$, and let $\sigma$ be the map that swaps the second and third arguments, so that $\sigma u$ stands for the quadruple $(x,z,y,t)$, and similarly $\sigma u'$. Thus we have
\[
S(u)=b(u)\,b(\sigma u)
\]
for each quadruple $u$. Let us further introduce the maps $\alpha,\beta,\gamma,\delta$ acting on ordered pairs of quadruples of objects in $\Cc$ by
\begin{align*}
\alpha(u,u')&=(x,z,x',z'),
\\
\beta(u,u')&=(y,t,y',t'),
\\
\gamma(u,u')&=(x,y,x',y'),
\\
\delta(u,u')&=(z,t,z',t')
\end{align*}
if $u=(x,y,z,t)$ and $u'=(x',y',z',t')$. Then according to the $4\times 4$ axiom we have (the product $u\cdot u'$ of quadruples $u,u'$ of objects is defined componentwise)
\begin{align*}
b(u\cdot u')\,[b(\gamma(u,u'))\cdot&b(\delta(u,u'))]\,b(\alpha(u,u')\cdot\beta(u,u'))
\\
&=[b(\alpha(u,u'))\cdot b(\beta(u,u'))]\,b(\gamma(u,u')\cdot \delta(u,u'))\,[b(u)\cdot b(u')]
\end{align*}
for every quadruples $u,u'$. In particular, this is still true when $u,u'$ are replaced by $\sigma u,\sigma u'$. Now, it is easy to check that
\begin{align*}
\alpha(\sigma u,\sigma u')&=\gamma(u,u'),
\\
\beta(\sigma u,\sigma u')&=\delta(u,u'),
\\
\gamma(\sigma u,\sigma u')&=\alpha(u,u'),
\\
\delta(\sigma u,\sigma u')&=\beta(u,u').
\end{align*}
Hence the $4\times 4$ axiom for the quadruples $\sigma u,\sigma u'$ becomes the equality
\begin{align*}
b(\sigma(u\cdot u'))\,&[b(\alpha(u,u'))\cdot b(\beta(u,u'))]\,b(\gamma(u,u')\cdot\delta(u,u'))
\\
&=[b(\gamma(u,u'))\cdot b(\delta(u,u'))]\,b(\alpha(u,u')\cdot\beta(u,u'))\,[b(\sigma u)\cdot b(\sigma u')]
\end{align*}
where we have used that $\sigma u\cdot \sigma u'=\sigma(u\cdot u')$. Combining both equations we have
\begin{align*}
b(u\cdot u')\,&b(\sigma(u\cdot u'))\,[b(\alpha(u,u'))\cdot b(\beta(u,u'))]\,b(\gamma(u,u')\cdot\delta(u,u'))
\\
&=[b(\alpha(u,u'))\cdot b(\beta(u,u'))]\,b(\gamma(u,u')\cdot\delta(u,u'))\,[b(u)\cdot b(u')]\,[b(\sigma u)\cdot b(\sigma u')]
\end{align*}
or equivalently, using that $\cdot$ is functorial,
\begin{align*}
S(u\cdot u')\,&[b(\alpha(u,u'))\cdot b(\beta(u,u'))]\,b(\gamma(u,u')\cdot\delta(u,u'))
\\
&=[b(\alpha(u,u'))\cdot b(\beta(u,u'))]\,b(\gamma(u,u')\cdot\delta(u,u'))\,[S(u)\cdot S(u')].
\end{align*}
Taking now $u=(x,y,1,1)$ and $u'=(1,1,z,t)$ it follows that $S(x,y,z,t)$ is the conjugate of
\[
S(x,y,1,1)\cdot S(1,1,z,t)=[b(x,1,y,1)\,b(x,y,1,1)]\cdot[b(1,z,1,t)\,b(1,1,z,t)]
\]
by the isomorphism
\[
[b(x,1,1,z)\cdot b(y,1,1,t)]\,b(x1,y1,1z,1t)=b(x1,y1,1z,1t),
\]
where we have used the normalization axiom (acc3). By the naturality of the associo-commutator in its arguments applied to the appropriate left and right unitors we finally obtain for $S(x,y,z,t)$ the expression
\begin{align*}
S(x,y,z,t)&=[(r_x^{-1}\cdot l^{-1}_z)\cdot(r_y^{-1}\cdot l_t^{-1})]\,b(x,y,z,t)\,[(r_x\cdot r_y)\cdot (l_z\cdot l_t)]
\\
&\hspace{2truecm}[[b(x,1,y,1)\,b(x,y,1,1)]\cdot[b(1,z,1,t)\,b(1,1,z,t)]]
\\
&\hspace{3truecm}[(r_x^{-1}\cdot r^{-1}_y)\cdot(l_z^{-1}\cdot l_t^{-1})]\,b(x,y,z,t)^{-1}\,[(r_x\cdot l_z)\cdot(r_y\cdot l_t)]
\end{align*}
Now, it readily follows from the unital axioms (acc2) that 
\[
(r_x\cdot r_y)\,b(x,1,y,1)\,b(x,y,1,1)\,(r_x^{-1}\cdot r^{-1}_y)=id_{xy}
\]
\[
(l_z\cdot l_t)\,b(1,z,1,t)\,b(1,1,z,t)\,(l_z^{-1}\cdot l^{-1}_t)=id_{zt}
\]
Hence, by functoriality $S(x,y,z,t)$ is indeed the identity.
\end{proof}

The next two results provide additional commutative diagrams whose commutativity follows from the axioms. These will later be used to show that specifying an AC-category is equivalent to specifying a symmetric monoidal category (cf. Theorem~\ref{teorema_equivalencia} below).

\begin{proposition}\label{lema0}
In every AC-category $\CC=(\Cc,\cdot,1,b,l,r)$ the following diagrams commute:
\begin{equation}\label{eq0'}
\xymatrix{
(11)(x1)\ar[rr]^-{b(1,1,x,1)}\ar[d]_-{r_1\cdot r_x} && (1x)(11)\ar[d]^-{l_x\cdot l_1}
\\
1x\ar[r]_-{l_x} & x & x1\ar[l]^-{r_x} }\quad
\xymatrix{
(1x)(11)\ar[rr]^-{b(1,x,1,1)}\ar[d]_-{l_x\cdot l_1} && (11)(x1)\ar[d]^-{r_\cdot r_x}
\\
x1\ar[r]_-{r_x} & x & 1x\ar[l]^-{l_x} } 
\end{equation} 
Moreover, $l_1=r_1$. 
\end{proposition}

\begin{proof}
It is a consequence of the commmutative diagrams
\[
\xymatrix{
(11)(x1)\ar[rr]^-{b(1,1,x,1)}\ar[d]_-{r_1\cdot r_x} && (1x)(11)\ar[d]^-{id\cdot l_1}
\\
1x\ar[d]_-{l_x} && (1x)1\ar[ll]^-{r_{1x}}\ar[d]^-{l_x\cdot id}
\\
x && x1\ar[ll]^-{r_x} }\quad
\xymatrix{
(1x)(11)\ar[rr]^-{b(1,x,1,1)}\ar[d]_-{l_x\cdot l_1} && (11)(x1)\ar[d]^-{r_1\cdot id}
\\
x1\ar[d]_-{r_x} && 1(x1)\ar[ll]^-{l_{x1}}\ar[d]^-{id\cdot r_x}
\\
x && 1x\ar[ll]^-{l_x} }
\]
where the top subdiagrams commute because of the unital axiom (acc2), and the bottom ones by the naturality of $r_{1x}$ and $l_{x1}$ respectively applied to the morphisms $l_x$ and $r_x$. To prove the last assertion, notice that $r_{11}=r_1\cdot id_1$ by the naturality of $r_x$ applied to the morphism $r_1$. Taking $x=y=1$ in the left diagram of (\ref{eq00}) and using the normalization axiom and the functoriality of $\cdot$ it follows that $r_1\cdot r_1=r_1\cdot l_1$ and hence $r_1=l_1$.
\end{proof}

From now on, we shall denote by $d$ the isomorphism $l_1=r_1$. 

\begin{proposition}\label{lema1}
In every AC-category $\CC=(\Cc,\cdot,1,b,l,r)$ the following diagrams commute for every objects $x,y,z,t$:
\begin{equation}\label{eq1}
\xymatrix{
(x1)((y1)(zt))\ar[rr]^-{id\cdot(r_y\cdot id)}\ar[d]_-{id\cdot b(y,1,z,t)} && (x1)(y(zt))\ar[rr]^-{b(x,1,y,zt)} && (xy)(1(zt))\ar[d]^-{r^{-1}_{xy}\cdot l_{zt}}
\\
(x1)((yz)(1t))\ar[d]_-{id\cdot(id\cdot l_t)} &&&& ((xy)1)(zt)\ar[d]^-{b(xy,1,z,t)}
\\
(x1)((yz)t)\ar[d]_-{b(x,1,yz,t)} &&&& ((xy)z)(1t)\ar[d]^-{(id\cdot l^{-1}_z)\cdot id}
\\
(x(yz))(1t)\ar[rr]_-{(r^{-1}_x\cdot id)\cdot id} && ((x1)(yz))(1t)\ar[rr]_-{b(x,1,y,z)\cdot id} && ((xy)(1z))(1t)
}
\end{equation} 
\begin{equation}\label{eq21}
\xymatrix{
x((1y)(z1))\ar[rr]^-{r_x^{-1}\cdot(l_y\cdot r_z)}\ar[d]_-{id\cdot b(1,y,z,1)} 
&&
(x1)(yz)\ar[rr]^-{b(x,1,y,z)}
&&
(xy)(1z)\ar[d]^-{id\cdot (r_z^{-1}l_z)} 
\\
x((1z)(y1))\ar[d]_-{r_x^{-1}\cdot(l_z\cdot r_y)} 
&&&& 
(xy)(z1)\ar[d]^-{b(x,y,z,1)}
\\
(x1)(zy)\ar[rr]_-{b(x,1,z,y)}
&&
(xz)(1y)\ar[rr]_-{id\cdot(r_y^{-1}l_y)} 
&&
(xz)(y1) }
\end{equation} 
\begin{equation}\label{eq23}
\xymatrix{
(1(xy))(z1)\ar[rr]^-{l_{xy}\cdot id}\ar[d]_-{b(1,xy,z,1)}
&&
(xy)(z1)\ar[rr]^-{b(x,y,z,1)}
&&
(xz)(y1)\ar[d]^-{(l_x^{-1}\cdot r_z^{-1})\cdot r_y}
\\
(1z)((xy)1)\ar[d]_-{(r_z^{-1}l_z)\cdot r_{xy}}
&&&&
((1x)(z1))y\ar[d]^-{b(1,x,z,1)\cdot id}
\\
(z1)(xy)\ar[rr]_-{b(z,1,x,y)} 
&& 
(zx)(1y)\ar[rr]_-{(l_z^{-1}\cdot r_x^{-1})\cdot l_y} 
&&
((1z)(x1)y
}
\end{equation} 
\begin{equation}\label{eq22}
\xymatrix{
(xy)(zt)\ar[rr]^-{r_{xy}^{-1}\cdot id}\ar[d]_-{b(x,y,z,t)} 
&&
((xy)1)(zt)\ar[rr]^-{b(xy,1,z,t)}
&&
((xy)z)(1t)\ar[d]^-{(id\cdot r_z^{-1})\cdot id)} 
\\
(xz)(yt)\ar[d]_-{r_{xz}^{-1}\cdot id} 
&&&& 
((xy)(z1))(1t)\ar[d]^-{b(x,y,z,1)\cdot id}
\\
((xz)1)(yt)\ar[rr]_-{b(xz,1,y,t)}
&&
((xz)y)(1t)\ar[rr]_-{(id\cdot r_y^{-1})\cdot id} 
&&
((xz)(y1))(1t) }
\end{equation} 
\end{proposition}

\begin{proof}
The commutativity of (\ref{eq1}) follows from the commutative diagram
{\small
\[
\xymatrix@C=4pc@R=2.7pc{
&
(x1)(y(zt))\ar[r]^-{b(x,1,y,zt)}\ar@{}@<-5ex>[r]|{(N)} 
&
(xy)(1(zt))\ar@{=}[rd]
&
\\
(x1)((y1)(zt))\ar[d]|{id\cdot b(y,1,z,t)}\ar[ru]^{id\cdot(r_y\cdot id)}
&
((x1)(11))((y1)(zt))\ar[r]_-{b(x1,11,y1,zt)}\ar[u]|-{(r_x\cdot d)\cdot(r_y\cdot id)}\ar[dd]|-{id\cdot b(y,1,z,t)} \ar[l]^{(r_x\cdot d)\cdot id} 
&
((x1)(y1))((11)(zt))\ar[dd]|-{b(x,1,y,1)\cdot b(1,1,z,t)}\ar[r]_-{(r_x\cdot r_y)\cdot(d\cdot id)} \ar[u]|-{(r_x\cdot r_y)\cdot(d\cdot id)} \ar@{}@<-6ex>[r]|{(2)} 
&
(xy)(1(zt))\ar[d]|-{r^{-1}_{xy}\cdot l_{zt}}
\\
(x1)((yz)(1t))\ar[d]|{id\cdot(id\cdot l_t)}
&&& 
((xy)1)(zt)\ar@{=}[dd]
\\
(x1)((yz)t)\ar[dd]|-{b(x,1,yz,t)}\ar@{}@<-9ex>[r]|{(N)} 
&
((x1)(11))((yz)(1t))\ar[dd]|-{b(x1,11,yz,1t)} \ar[l]_-{(r_x\cdot d)\cdot(id\cdot l_t)} \ar@{}[r]|{\ \  (1)} \ar[lu]|{(r_x\cdot d)\cdot id}
&
((xy)(11))((1z)(1t))\ar[dd]|{b(xy,11,1z,1t)}\ar[ur]|-{(id\cdot d)\cdot(l_z\cdot l_t)}\ar[dr]|-{(id\cdot d)\cdot(l_z\cdot l_t)}
&
\\ 
&&
\ar@{}@<-2ex>[r]|{(N)} 
& ((xy)1)(zt)\ar[d]|-{b(xy,1,z,t)}
\\
(x(yz))(1t)\ar[rd]_{(r_x^{-1}\cdot id)\cdot id}
&
((x1)(yz))((11)(1t))\ar[r]_-{b(x,1,y,z)\cdot id} \ar[l]^-{(r_x\cdot id)\cdot(d\cdot l_t)}\ar[d]|{id\cdot(d\cdot l_t)}
&
((xy)(1z))((11)(1t))\ar[r]_-{(id\cdot l_z)\cdot(d\cdot l_t)}\ar[d]|{id\cdot(d\cdot l_t)}
&
((xy)z)(1t)\ar[ld]^{(id\cdot l^{-1}_z)\cdot id}
\\
&
((x1)(yz))(1t)\ar[r]_{b(x,1,y,z)\cdot id}
&
((xy)(1z))(1t)
&
}
\] }

\noindent 
where we have already used the normalization axiom (acc3) to write the appropriate identity morphisms instead of $b(x,1,1,1)$ and $b(1,1,1,t)$ in the middle hexagon. In this diagram the subdiagrams labelled (N) are naturality squares, those labelled (2) commute by the unital axioms (acc2), the subdiagram labelled (1) by the $4\times 4$ axiom (acc1) with $y=z=t=y'=1$, $x'=y$, $z'=z$ and $t'=t$, and all the unlabelled subdiagrams by the functoriality of the product. Since the outer diagram coincides with (\ref{eq1}), it follows that (\ref{eq1}) commutes.

The commutativity of (\ref{eq21}) follows from the commutative diagram
{\small
\[
\xymatrix@C=4pc@R=2.7pc{
x((1y)(z1))\ar[r]^{r_x^{-1}\cdot(l_y\cdot r_z)}\ar[d]|{r^{-1}_x\cdot id\ }
& 
(x1)(yz)\ar[r]^-{b(x,1,y,z)}\ar@{}@<-5ex>[r]|{(N)} 
& 
(xy)(1z)\ar[rd]^{(r_x^{-1}\cdot l_y^{-1})\cdot id}
& 
\\
(x1)((1y)(z1))\ar[d]|{id\cdot b(1,y,z,1)}\ar[ru]|{id\cdot(l_y\cdot r_z)}
&
((x1)(11))((1y)(z1))\ar[r]^{b(x1,11,1y,z1)}\ar[u]|-{(r_{x}\cdot d)\cdot(l_y\cdot r_z)}\ar[dd]|-{id\cdot b(1,y,z,1)}\ar[l]^{(r_x\cdot d)\cdot id}
&
((x1)(1y))((11)(z1))\ar[u]|-{(r_x\cdot l_y)\cdot(d\cdot r_z)}\ar[dd]|-{id\cdot b(1,1,z,1)}\ar[r]^-{id\cdot(d\cdot r_z)} \ar@{}@<-5ex>[r]|{(2)}
&
((x1)(1y))(1z)\ar[d]|-{id\cdot(r^{-1}_z l_z)} 
\\
(x1)((1z)(y1))\ar[d]|{id\cdot(l_z\cdot r_y)}
&&&
((x1)(1y))(z1)\ar[d]|{(r_x\cdot l_y)\cdot id}  
\\
(x1)(zy)\ar[d]|-{b(x,1,z,y)}  \ar@{}@<-6ex>[r]|{(N)} 
& 
((x1)(11))((1z)(y1))\ar[l]_-{(r_x\cdot d)\cdot(l_z\cdot r_y)}\ar[d]|-{b(x1,11,1z,y1)} \ar@{}@<5ex>[r]|{(1)} \ar[lu]|{(r_x\cdot d)\cdot id}
& 
((x1)(1y))((1z)(11))\ar[d]|-{b(x1,1y,1z,11)}\ar[r]^-{(r_x\cdot l_y)\cdot(l_z\cdot d)}\ar[ru]|-{id\cdot(l_z\cdot d)} \ar@{}@<-6ex>[r]|{(N)} 
&
(xy)(z1)\ar[d]|-{b(x,y,z,1)}
\\
(xz)(1y)\ar[rd]_{(r^{-1}_x\cdot l^{-1}_z)\cdot id}
& 
((x1)(1z))((11)(y1))\ar[l]_-{(r_x\cdot l_z)\cdot(d\cdot r_y)}\ar[r]^-{id\cdot b(1,1,y,1)}\ar[d]|-{id\cdot(d\cdot r_y)}\ar@{}@<-7ex>[r]|{(2)} 
&  
((x1)(1z))((1y)(11))\ar[r]^-{(r_x\cdot l_z)\cdot(l_y\cdot d)}\ar[d]|-{id\cdot(l_y\cdot d)}
& 
(xz)(y1)
\\
&  
((x1)(1z))(1y)\ar[r]_-{id\cdot(r_y^{-1}l_y)}
&
((x1)(1z))(y1)\ar[ru]_{(r_x\cdot l_z)\cdot id}
&
 }
\] } 

\noindent
where we have already used the normalization axiom (acc3) to write the appropriate identity morphisms instead of $b(x,1,1,y)$, $b(x,1,1,z)$ and $b(x,1,1,1)$ in the middle hexagon. In this diagram the subdiagrams labelled (N) are naturality squares, those labelled (2) commute by Proposition~\ref{lema0}, the subdiagram labelled (1) by the $4\times 4$ axiom (acc1) with $y=z=t=x'=t'=1$ and $y'=y$, $z'=z$, and all unlabelled subdiagrams by the functoriality of the product. Now, again by functoriality of the product, we have that:
\begin{itemize}
\item[(1)] the composite of the three morphisms at the bottom is
\[
[(r_x\cdot l_z)\cdot id]\,[id\cdot(r^{-1}_y\cdot l_y)]\,[(r^{-1}_x\cdot l^{-1}_z)\cdot id]=id\cdot(r_y^{-1}\cdot l_y);
\]
\item[(2)] the composite of the first three morphisms at right is
\[
[(r_x\cdot l_y)\cdot id]\,[id\cdot(r^{-1}_z\cdot l_z)]\,[(r^{-1}_x\cdot l^{-1}_y)\cdot id]=[id\cdot(r_z^{-1}\cdot l_z)];
\]
\item[(3)] the composite of the first three morphisms at left is 
\[
[id\cdot(l_z\cdot r_y)]\,(id\cdot b(1,y,z,1))\,(r_x^{-1}\cdot id)=[r_x^{-1}\cdot(l_z\cdot r_y)]\,(id\cdot b(1,y,z,1)).
\]
\end{itemize}
Therefore, the outer diagram indeed coincides with (\ref{eq21}), and (\ref{eq21}) commutes.

The commutativity of (\ref{eq23}) follows from the commutative diagram
{\small
\[
\xymatrix@C=3.2pc@R=3pc{
&
(1(xy))(z1)\ar[r]^-{b(1,xy,z,1)} \ar@{}@<-5ex>[r]|{(N)} 
&
(1z)((xy)1)
&
\\
(1(xy))((z1)(11))\ar[d]|-{l_{xy}\cdot id} \ar@{}@<-5ex>[r]|{(2)} 
&
((11)(xy))((z1)(11))\ar[l]_-{(d\cdot id)\cdot id}\ar[r]^-{b(11,xy,z1,11)}\ar[d]|-{b(1,1,x,y)\cdot id}
\ar[u]^-{(d\cdot id)\cdot(r_z\cdot d)}
&
((11)(z1))((xy)(11))\ar[d]|-{b(1,1,z,1)\cdot b(x,y,1,1)}\ar[u]_-{(d\cdot r_z)\cdot(id\cdot d)}\ar[r]^-{(d\cdot r_z)\cdot(id\cdot d)}\ar@{}@<-5ex>[r]|{(4)} 
& 
(1z)((xy)1)\ar[d]|-{(r^{-1}_zl_z)\cdot r_{xy}}
\\
(xy)((z1)(11))
&
((1x)(1y))((z1)(11))\ar[ld]|-{(l_x\cdot l_y)\cdot(r_z\cdot d)}\ar[dd]|-{b(1x,1y,z1,11)}\ar[l]_-{(l_x\cdot l_y)\cdot id}\ar@{}@<-9ex>[r]|{(1)} 
& 
((1z)(11))((x1)(y1))\ar[dd]|-{b(1z,11,x1,y1)} \ar[rd]|-{(l_z\cdot d)\cdot(r_x\cdot r_y)}\ar[r]_-{(l_z\cdot d)\cdot(r_x\cdot r_y)}
& (z1)(xy)
\\
(xy)(z1)\ar[d]|-{b(x,y,z,1)} \ar@{}@<-2ex>[r]|{(N)} 
&&&
(z1)(xy)\ar[d]|-{b(z,1,x,y)} \ar@{}@<2ex>[l]|{(N)} 
\\
(xz)(y1)
&
((1x)(z1))((1y)(11))\ar[r]^-{b(1,x,z,1)\cdot b(1,y,1,1)} \ar[l]^-{(l_x\cdot r_z)\cdot(l_y\cdot d)}\ar[d]_-{id\cdot (l_y\cdot d)}\ar@{}@<-5ex>[r]|{(3)} 
&
((1z)(x1))((11)(y1))\ar[r]_-{(l_z\cdot r_x)\cdot(d\cdot r_y)} \ar[d]^-{id\cdot(d\cdot r_y)}
&
(zx)(1y)
\\
&
((1x)(z1))(y1)\ar[r]_-{id\cdot (l^{-1}_yr_y)} 
&
((1x)(z1))(1y)
&
}
\] }

\noindent
where we have already used the normalization axiom (acc3) to write an identity morphism instead of $b(z,1,1,1)$ in the middle hexagon. In this diagram the subdiagrams labelled (N) are naturality squares, those labelled (2), (3), (4) commute by the unital axioms (acc2) and/or Proposition~\ref{lema0}, and the subdiagram labelled (1) commutes by the $4\times 4$ axiom (acc1) with $x=y=y'=z'=t'=1$, $z=x$, $t=y$ and $x'=z$. As before, using the functoriality of the product, it is easy to check that several left and right unitors cancel each other, so that the outer diagram indeed coincides with (\ref{eq23}).

The commutativity of (\ref{eq22}) follows from the commutative diagram
{\small
\[
\xymatrix@C=3.2pc@R=3pc{
&
((xy)1)(zt)\ar[r]^-{b(xy,1,z,t)} \ar@{}@<-5ex>[r]|{(N)} 
&
((xy)z)(1t)
& 
\\
((xy)1)((z1)(1t))\ar[d]|-{r_{xy}\cdot id} \ar@{}@<-5ex>[r]|{(2)} 
&
((xy)(11))((z1)(1t))\ar[l]_-{(id\cdot d)\cdot id}\ar[r]^-{b(xy,11,z1,1t)}\ar[d]|-{b(x,y,1,1)\cdot id}
\ar[u]^-{(id\cdot d)\cdot(r_z\cdot l_t)}
&
((xy)(z1))((11)(1t))\ar[d]|-{b(x,y,z,1)\cdot id}\ar[u]_-{(id\cdot r_z)\cdot(d\cdot l_t)}
& 
\\
(xy)((z1)(1t))
&
((x1)(y1))((z1)(1t))\ar[ld]|-{(r_x\cdot r_y)\cdot(r_z\cdot l_t)}\ar[dd]|-{b(x1,y1,z1,1t)}\ar[l]_-{(r_x\cdot r_y)\cdot id}\ar@{}@<-9ex>[r]|{(1)} 
& 
((xz)(y1))((11)(1t))\ar[dd]|-{b(xz,y1,11,1t)} \ar[rd]|-{(id\cdot r_y)\cdot(d\cdot l_t)}
& 
\\
(xy)(zt)\ar[d]|-{b(x,y,z,t)} \ar@{}@<-2ex>[r]|{(N)} 
&&&
((xz)y)(1t)\ar[d]|-{b(xz,y,1,t)} \ar@{}@<2ex>[l]|{(N)} 
\\
(xz)(yt)
&
((x1)(z1))((y1)(1t)))\ar[r]^-{b(x,1,z,1)\cdot id} \ar[l]^-{(r_x\cdot r_z)\cdot(r_y\cdot l_t)}\ar[d]_-{(r_x\cdot r_z)\cdot id}\ar@{}@<-5ex>[r]|{(2)} 
&
((xz)(11))((y1)(1t))\ar[r]_-{(id\cdot d)\cdot (r_y\cdot l_t)} \ar[d]^-{(id\cdot d)\cdot id}
&
((xz)1)(yt)
\\
&
(xz)((y1)(1t))\ar[r]_-{r^{-1}_{xz}\cdot id} 
&
((xz)1)((y1)(1t))
&
}
\] }

\noindent
where we have already used the normalization axiom (acc3) to write the appropriate identity morphisms instead of $b(y,1,1,t)$, $b(z,1,1,t)$ and $b(1,1,1,t)$ in the middle hexagon. In this diagram, the subdiagrams labeled (N) are naturality squares, those labeled (2) commute by the unital axioms (acc2) and the functoriality of the product, and the subdiagram labeled (1) commutes by the $4\times 4$ axiom (acc1) with $z=t=y'=z'=1$, $x'=z$, and $t'=t$. Once more, using the functoriality of the product, it is easy to check that several left and right unitors cancel each other, so that the outer diagram indeed coincides with (\ref{eq22}).
\end{proof}

\begin{definition}
Given two AC-categories $\CC=(\Cc,\cdot,1,b,l,r)$ and $\CC'=(\Cc',\cdot',1',b',l',r')$, an {\em AC-functor} from $\CC$ to $\CC'$ is a triple $\FF=(F,F_2,F_1)$ consisting of:
\begin{itemize}
\item a functor $F:\Cc\to\Cc'$;
\item natural isomorphisms $F_2(x,y):FxFy\to F(xy)$ for every objects $x,y$ in $\Cc$;
\item an isomorphism $F_1:1'\to F1$
\end{itemize}
(for short, the $\cdot$ between objects is omitted). Moreover, these data must satisfy the following axioms (the product of morphisms is denoted by $\cdot$ in both AC-categories):
\begin{itemize}
\item[{\rm (acf1)}] for every objects $x,y,z,t$ in $\Cc$ the following diagram commutes:
{\small
\[
\xymatrix{
(FxFy)(FzFt)\ar[rrr]^-{b'(Fx,Fy,Fz,Ft)}\ar[d]_-{F_2(x,y)\cdot F_2(z,t)}
&&&
(FxFz)(FyFt)\ar[d]^-{F_2(x,z)\cdot F_2(y,t)} 
\\
F(xy)F(zt)\ar[d]_-{F_2(xy,zt)}
&&&
F(xz)F(yt)\ar[d]^-{F_2(xz,yt)}
\\
F((xy)(zt))\ar[rrr]_-{Fb(x,y,z,t)} &&& F((xz)(yt)) }
\] }
\item[{\rm (acf2)}] for ever object $x$ in $\Cc$ the following diagrams commute:
{\small
\[
\xymatrix{
1'Fx\ar[r]^-{l'_{Fx}}\ar[d]_-{F_1\cdot id} & Fx
\\
F1Fx\ar[r]_-{F_2(1,x)} & F(1x)\ar[u]_-{Fl_x} }
\quad
\xymatrix{
Fx1'\ar[r]^-{r'_{Fx}}\ar[d]_-{id\cdot F_1} & Fx
\\
FxF1\ar[r]_-{F_2(x,1)} & F(x1)\ar[u]_-{Fr_x} }
\] }
\end{itemize}
$\FF$ called a {\em unital AC-functor} when $F_1$ is the identity (in particular, $F1=1'$), and a {\em strict AC-functor} when it is unital and $F_2(x,y)$ is the identity for every objects $x,y$ in $\Cc$ (in particular, $F$ preserves the product on the nose).
\end{definition}

It readily follows that the identity functor $id_\Cc:\Cc\to\Cc$ is a strict AC-functor $\mathbb{id}_\CC=(id_\Cc,1,id_1)$ for every AC-category $\CC$. Observe that the data specifying an AC-functor between AC-categories is exactly the same as the data specifying a symmetric monoidal functor between symmetric monoidal categories. We will use the same notation for both, and the context will make clear whether it refers to an AC-functor or a symmetric monoidal functor. Next result shows that the composition law is also the same for both types of functors.

\begin{proposition}\label{composicio_AC-functors}
Let $\FF:\CC\to\DD$ and $\GG:\DD\to\EE$ be AC-functors, with $\FF=(F,F_2,F_1)$ and $\GG=(G,G_2,G_1)$. Then the composite functor $G\circ F$ is canonically an AC-functor with the isomorphisms $(G\circ F)_2(x,y)$ and $(G\circ F)_1$ given respectively by the composites
\begin{equation}\label{estructura_ac_F'circF_1}
\xymatrix{
GFxGFy\ar[rr]^-{G_2(Fx,Fy)} && G(FxFy)\ar[rr]^-{G(F_2(x,y))} && GF(xy)}
\end{equation}
\begin{equation}\label{estructura_ac_F'circF_2}
\xymatrix{
1_e\ar[r]^-{G_1} & G1_d\ar[r]^-{G(F_1)} & GF1_c}
\end{equation}
Moreover, this composition is strictly associative and has the identity AC-functors $\mathbb{id}_\CC$ as strict units.
\end{proposition}

\begin{proof}
Axiom (acf1) on $(G\circ F)_2$ amounts to the commutativity of the diagram
{\footnotesize
\[
\xymatrix@R=3pc@C=0.5pc{
\mbox{\bf ((GFx)(GFy))((GFz)(GFt))}\ar[rrrr]^-{b_e(GFx,GFy,GFz,GFt)}\ar[d]|-{G_2(Fx,Fy)\cdot G_2(Fz,Ft)}\ar@{}@<-12ex>[rrrr]^{(A)}
&&&&
\mbox{\bf ((GFx)(GFz))((GFy)(GFt))}\ar[d]|-{G_2(Fx,Fz)\cdot G_2(Fy,Ft)} 
\\
G(FxFy)G(FzFt)\ar[d]|-{G(F_2(x,y))\cdot G(F_2(z,t))}\ar[rd]^-{\ \ \ \ \  G_2(FxFy,FzFt)}
&&&&
G(FxFz)G(FyFt)\ar[d]|-{G(F_2(x,z))\cdot G(F_2(y,t))}\ar[ld]_-{G_2(FxFz,FyFt)\ \ \ \ \ }
\\
\mbox{\bf G(F(xy))G(F(zt))}\ar[d]|-{G_2(F(xy),F(zt))} \ar@{}@<-2ex>[r]^{(N)}
& 
G(F(xy)F(zt))\ar[ld]^{\ \ \ \ \ \ \ G(F_2(x,y)\cdot F_2(z,t))}\ar@/_0.2pc/[rr]_-{Gb_d(Fx,Fy,Fz,Ft)} \ar@{}@<-12ex>[rr]^{(B)}
&& 
G((FxFz)(FyFt))\ar[rd]_-{G(F_2(x,z)\cdot F_2(y,t))\ \ \ \ \ \ \  } \ar@{}@<-2ex>[r]^{(N)}
&
\mbox{\bf G(F(xz))G(F(yt))}\ar[d]|-{G_2(F(xz),F(yt))}
\\
G(F(xy)F(zt))\ar[d]|-{G(F_2(xy,zt))} 
&&&&
G(F(xz)F(yt))\ar[d]|-{G(F_2(xz,yt))}
\\
\mbox{\bf GF((xy)(zt))}\ar[rrrr]_-{GF\,b_c(x,y,z,t)} &&&& \mbox{\bf GF((xz)(yt))} 
}
\] }

\noindent
and in this diagram subdiagram (A) commutes because of (acf1) applied to $\GG$, subdiagram (B) because of (acf1) applied to $\FF$ and by functoriality, and the subdiagrams labelled (N) are naturality squares. The proofs of (acf2) and of the last statement are the same as those for monoidal functors between monoidal categories.
\end{proof}

\begin{definition}
Let $\FF,\GG:\CC\to\CC'$ be two parallel AC-functors, with  $\FF=(F,F_2,F_1)$ and $\GG=(G,G_2,G_1)$. An {\em AC-natural transformation} from $\FF$ to $\GG$ is a natural transformation $\tau:F\Rightarrow G$ such that the diagrams
\begin{equation}\label{AC-trans_natural}
\xymatrix{
FxFy\ar[rr]^-{F_2(x,y)}\ar[d]_-{\tau_x\cdot\tau_y} && F(xy)\ar[d]^-{\tau_{xy}}
\\
GxGy\ar[rr]_-{G_2(x,y)} && F(xy) }
\qquad
\xymatrix{
1'\ar[rr]^{F_1}\ar[dr]_-{G_1} && F1\ar[dl]_-{\tau_1}
\\
& G1 & }
\end{equation}
commute for every objects $x,y$ in $\Cc$.
\end{definition}

It readily follows that the identity natural transformation $1_F:F\Rightarrow F$ of the underlying functor of every AC-functor $\FF=(F,F_2,F_1)$ is an AC-natural transformation $1_\mathbb{F}$ for every AC-functor $\FF$.

\begin{proposition}\label{composicions_AC-transformacions}
The usual vertical and horizontal composites of AC-natural transformations are AC-natural transformations.
\end{proposition}

\begin{proof}
AC-natural transformations are defined in exactly the same way as the monoidal natural transformations between monoidal functors, and these are known to be closed under vertical and horizontal compositions.
\end{proof}

It follows that the AC-categories together with the AC-functors and the AC-natural transformations consitute a strict 2-category $\mathbf{ACCat}$ with the above composition laws.

\begin{theorem}\label{teorema_equivalencia}
Let $\mathbf{SMCat}$ be the 2-category of symmetric monoidal categories, symmetric monoidal functors and monoidal natural transformations. Then there is a canonical isomorphism of 2-categories $\mathbf{ACCat}\cong\mathbf{SMCat}$.
\end{theorem}

\begin{proof}
Every symmetric monoidal category $\CC=(\Cc,\cdot,1,a,c,l,r)$ gives rise to an AC-category $\CC_{ac}$ with the same data $\Cc,\cdot,1,l,r$ as in $\CC$ and associo-commutator $b(x,y,z,t)$ given from the associator and the commutator in $\CC$ by the composite isomorphism in the commutative diagram
\begin{equation}\label{associocommutador}
\xymatrix@C=3pc{
&
x(y(zt))\ar[r]^-{id\cdot a_{y,z,t}} 
&
x((yz)t)\ar[r]^-{id\cdot(c_{y,z}\cdot id)}\ar[dd]|-{a_{x,yz,t}} 
&
x((zy)t)\ar[r]^-{id\cdot a^{-1}_{z,y,t}}\ar[dd]|-{a_{x,zy,t}}
&
x(z(yt))\ar[dr]|-{a_{x,z,yt}}
&
\\
(xy)(zt)\ar[ur]|-{a^{-1}_{x,y,zt}}\ar[dr]|-{a_{xy,z,t}}
&&&&&
(xz)(yt)
\\
&
((xy)z)t\ar[r]_-{a^{-1}_{x,y,z}\cdot id}
& 
(x(yz))t\ar[r]_-{(id\cdot c_{y,z})\cdot id}
&
(x(zy))t\ar[r]_-{a_{x,z,y}\cdot id}
&
((xz)y)t\ar[ur]|-{a^{-1}_{xz,y,t}} 
&
}
\end{equation}
Axioms (acc1)-(acc3) on $b(x,y,z,t)$ follow from the coherence theorem for symmetric monoidal categories ( see \cite[Theorem~XI.1.1]{MacLane-1998} or \cite[Theorem~1.3.8]{Yau-2024-I}). Moreover, every symmetric monoidal functor $\FF=(F,F_2,F_0)$ between symmetric monoidal categories $\CC,\CC'$ is also an AC-functor between the corresponding AC-categories $\CC_{ac},\CC'_{ac}$. Axiom (acf2) holds automatically by definition of symmetric monoidal functor, while axiom (acf1) follows from the coherence theorem for symmetric monoidal functors (see \cite[Theorem~1.3.12]{Yau-2024-I} or \cite[Theorem~1.7]{Gurski-Johnson-2025}). Finally, every monoidal natural transformation between symmetric monoidal functors is automatically an AC-natural transformation between them as AC-functors. Therefore we have a strict 2-functor
\[
(-)_{ac}:\mathbf{SMCat}\to\mathbf{ACCat}
\]
given on objects by $\CC\mapsto\CC_{ac}$ and acting as the identity on morphisms and 2-morphisms.

Conversely, given an AC-category $\DD=(\Dd,\cdot,1,b,l,r)$ a symmetric monoidal category $\DD_{sm}$ is given as follows. The data $\Dd,\cdot,1,l,r$ is the same as in $\DD$, and the associators $a_{x,y,z}$ and commutators $c_{x,y}$ are respectively given by the composite isomorphisms
\begin{equation}\label{associador_b}
\xymatrix@C=3pc{
x(yz)\ar[r]^-{r_x^{-1}\cdot id} & (x1)(yz)\ar[r]^-{b(x,1,y,z)} & (xy)(1z)\ar[r]^-{id\cdot l_z} & (xy)z }
\end{equation}
\begin{equation}\label{commutador_b}
\xymatrix@C=3pc{
xy\ar[r]^-{l_x^{-1}\cdot r_y^{-1}} & (1x)(y1)\ar[r]^-{b(1,x,y,1)} & (1y)(x1)\ar[r]^-{l_y\cdot r_x} & yx }
\end{equation}
We have to see that the associator satisfies the pentagon and triangle axioms, and the commutator the hexagon and symmetry axioms. The pentagon axiom on $a_{x,y,z}$ corresponds to the commutativity of the diagram
{\small
\[
\xymatrix@C=3.5pc{
\mbox{\bf x(y(zt))}\ar[r]^-{r_x^{-1}\cdot id}\ar[d]_-{id\cdot(r_y^{-1}\cdot id)} 
& 
(x1)(y(zt))\ar[rr]^-{b(x,1,y,zt)}
&& 
(xy)(1(zt))\ar[d]^-{id\cdot l_{zt}} 
\\ 
x((y1)(zt))\ar[d]_-{id\cdot b(y,1,z,t)} 
&&& 
\mbox{\bf (xy)(zt)}\ar[d]^-{r^{-1}_{xy}\cdot id}
\\
x((yz)(1t))\ar[d]_-{id\cdot (id\cdot l_t)}
&&& 
((xy)1)(zt)\ar[d]^-{b(xy,1,z,t)}
\\
\mbox{\bf x((yz)t)}\ar[d]_-{r_x^{-1}\cdot id} 
&&&  
((xy)z)(1t)\ar[d]^-{id\cdot l_t} 
\\
(x1)((yz)t)\ar[d]_-{b(x,1,yz,t)} 
&&&   
\mbox{\bf ((xy)z)t}
\\
(x(yz))(1t)\ar[r]_-{id\cdot l_t} 
& 
\mbox{\bf (x(yz))t}\ar[r]_-{(r_x^{-1}\cdot id)\cdot id}  
& 
((x1)(yz))t\ar[r]_-{b(x,1,y,z)\cdot id} 
& 
((xy)(1z))t\ar[u]_-{(id\cdot l_z)\cdot id}
}
\] }

\noindent
or equivalently, the diagram
{\small
\[
\xymatrix@C=3.5pc{
x((y1)(zt))\ar[d]_-{id\cdot b(y,1,z,t)} \ar[r]^-{r^{-1}_x\cdot (r_y\cdot id)}
& 
(x1)(y(zt))\ar[r]^-{b(x,1,y,zt)}
& 
(xy)(1(zt))\ar[d]^-{r^{-1}_{xy}\cdot l_{zt}}
\\
x((yz)(1t))\ar[d]_-{r^{-1}_x\cdot (id\cdot l_t)} 
&& 
((xy)1)(zt)\ar[d]^-{b(xy,1,z,t)}
\\
(x1)((yz)t)\ar[d]_-{b(x,1,yz,t)} 
&&  
((xy)z)(1t)\ar[d]^-{(id\cdot l^{-1}_{z})\cdot l_{t}}
\\
(x(yz))(1t)\ar[r]_-{(r^{-1}_{x}\cdot id)\cdot l_{t}} 
& 
((x1)(yz))t\ar[r]_-{b(x,1,y,z)\cdot id} 
& 
((xy)(1z))t
}
\] }

\noindent
where we have used the functoriality of $\cdot$ to replace some composites involving left and right unitors by a single arrow. Apart from the isomorphisms $r^{-1}_x$ in the upper left and $l_t$ in the lower right, which may be cancelled, this diagram coincides with the commutative diagram (\ref{eq1}) in Proposition~\ref{lema1}. The triangle axiom amounts to the commutativity of the diagram
\[
\xymatrix@C=4pc{
\mbox{\bf x(1y)}\ar[r]^-{r^{-1}_x\cdot id}\ar[d]_-{id\cdot l_y} & (x1)(1y)\ar[r]^-{b(x,1,1,y)} & (x1)(1y)\ar[d]^-{id\cdot l_y}
\\
\mbox{\bf xy} && \mbox{\bf (x1)y}\ar[ll]^-{r_x\cdot id}
}
\]
whose commutativity follows from the functoriality of $\cdot$ and the normalization axiom (acc3). The hexagon axiom on $c_{x,y}$ amounts to the commutativity of the diagram
{\small
\[
\xymatrix@C=3pc{
\mbox{\bf x(yz)}\ar[r]^-{r_x^{-1}\cdot id}\ar[d]_{id\cdot (l_y^{-1}\cdot r_z^{-1})} 
&
(x1)(yz)\ar[r]^-{b(x,1,y,z)}  
& 
(xy)(1z)\ar[r]^-{id\cdot l_z}
&
\mbox{\bf (xy)z}\ar[r]^-{l^{-1}_{xy}\cdot r_z^{-1}}
&
(1(xy))(z1)\ar[d]^-{b(1,xy,z,1)}
\\
x((1y)(z1))\ar[d]_-{id\cdot b(1,y,z,1)}
&&&& 
(1z)((xy)1)\ar[d]^-{l_z\cdot r_{xy}}
\\
x((1z)(y1))\ar[d]_-{id\cdot(l_z\cdot r_y)}
&&&& 
\mbox{\bf z(xy)}\ar[d]^-{r_z^{-1}\cdot id}
\\
\mbox{\bf x(zy)}\ar[d]_-{r_x^{-1}\cdot id}  
&&&& 
(z1)(xy)\ar[d]^-{b(z,1,x,y)}
\\ 
(x1)(zy)\ar[d]_-{b(x,1,z,y)} 
&&&& 
(zx)(1y)\ar[d]_-{id\cdot l_y}
\\
(xz)(1y)\ar[r]_-{id\cdot l_y}
&
\mbox{\bf (xz)y}\ar[r]_-{(l_x^{-1}\cdot r_z^{-1})\cdot id} 
&
((1x)(z1))y\ar[r]_-{b(1,x,z,1)\cdot id} 
&
((1z)(x1))y\ar[r]_-{(l_z\cdot r_x)\cdot id}
&
\mbox{\bf (zx)y}
}
\] }

\noindent
or equivalently, the diagram
{\small
\[
\xymatrix@C=4pc{
x((1y)(z1))\ar[d]_-{id\cdot b(1.y.z.1)}\ar[r]^-{r^{-1}_x\cdot(l_y\cdot r_z)}
&
(x1)(yz)\ar[r]^-{b(x,1,y,z)}  
& 
(xy)(1z)\ar[r]^-{l^{-1}_{xy}\cdot (r^{-1}_zl_{z})}\ar@{.>}[ld]_-{id\cdot(r^{-1}_z\cdot l_z)}
&
(1(xy))(z1)\ar[d]^-{b(1,xy,z,1)}\ar@{.>}[lld]^-{l_{xy}\cdot id}
\\
x((1z)(y1))\ar[d]_-{r^{-1}_{x}\cdot (l_z\cdot r_{y})}\ar@{}@<-5ex>[r]^-{(A)}
&
(xy)(z1)\ar@{.>}[d]^-{b(x,y,z,1)}\ar@{}@<-5ex>[rr]^-{(B)}
&&
(1z)((xy)1)\ar[d]^-{(r^{-1}_{z}l_z)\cdot r_{xy}}
\\
(x1)(zy)\ar[d]_-{b(x,1,z,y)} 
&
(xz)(y1)\ar@{.>}[ld]_-{id\cdot(l^{-1}_y\cdot r_y)}\ar@{.>}[d]^-{(l^{-1}_x\cdot r^{-1}_z)\cdot r_y}
&&
(z1)(xy)\ar[d]^-{b(z,1,x,y)}
\\ 
(xz)(1y)\ar[r]_-{(l^{-1}_x\cdot r^{-1}_z)\cdot l_{y}} 
&
((1x)(z1))y\ar[r]_-{b(1,x,z,1)\cdot id} 
&
((1z)(x1))y\ar[r]_-{(l_z\cdot r_x)\cdot l^{-1}_{y}} 
&
(zx)(1y)
}
\] }

\noindent
where we have used again the functoriality of $\cdot$ to replace a few composites morphisms involving left and right unitors by a single arrow. Now, as indicated by the dotted arrows, this diagram can be subdivided into two triangles together with the subdiagrams (A) and (B). The triangles commute by functoriality, while the diagrams (A) and (B) coincide with the commutative diagrams (\ref{eq21}) and (\ref{eq23}), respectively, in Proposition~\ref{lema1}. Hence the outer diagram also commutes. Finally, the symmetry axiom on $c_{x,y}$ follows readily from (\ref{commutador_b}) and Proposition~\ref{cond_simetria}. This proves that $\DD_{sm}$ indeed is a symmetric monoidal category. Moreover, every AC-functor $\FF=(F,F_2,F_0)$ between AC-categories $\DD,\DD'$ is also a symmetric monoidal functor between the corresponding symmetric monoidal categories $\DD_{sm},\DD'_{sm}$, and the same happens with the natural transformations. Therefore, we also have a (strict) 2-functor 
\[
(-)_{sm}:\mathbf{ACCat}\to\mathbf{SMCat}
\]
given on objects by $\DD\mapsto\DD_{sm}$ and acting as the identity on morphisms and 2-morphisms. 

It remains to check that both composite 2-functors $(-)_{sm}\circ(-)_{ac}$ and $(-)_{ac}\circ(-)_{sm}$ are identities. At the level of morphisms and 2-morphisms it is clearly so. Hence we simply need to check the equalities
\begin{align*}
(\CC_{ac})_{sm}&=\CC,
\\
(\DD_{sm})_{ac}&=\DD
\end{align*}
for every symmetric monoidal category $\CC$, and every AC-category $\DD$. This means checking that the associator and the commutator in $(\CC_{ac})_{sm}$ are the same as in $\CC$, and that the associo-commutator in $(\DD_{sm})_{ac}$ is the same as in $\DD$. It readily follows from (\ref{associocommutador})-(\ref{commutador_b}) that the associator and commutator in $(\CC_{ac})_{sm}$ are respectively given by the long composite isomorphisms shown in the diagrams
{\small
\[
\xymatrix@R=3pc{
x(yz)\ar[r]^-{r_x^{-1}\cdot id}\ar[d]_{a_{x,y,z}} 
& 
(x1)(yz)\ar[r]^{a_{x1,y,z}}\ar@{.>}[d]^{b(x,1,y,z)} 
& 
((x1)y)z\ar[r]^-{a^{-1}_{x,1,y}\cdot id} 
&
(x(1y))z\ar[d]^-{(id\cdot c_{1,y})\cdot id}
\\
(xy)z 
& 
(xy)(1z)\ar[l]^-{id\cdot l_z} 
& 
((xy)1)z\ar[l]^-{a^{-1}_{xy,1,z}}
&
(x(y1))z\ar[l]^-{a_{x,y,1}\cdot id}
}
\] 
}
{\small
\[
\xymatrix@R=3pc{
xy\ar[r]^-{l_x^{-1}\cdot r^{-1}_y}\ar[d]_{c_{x,y}} 
& 
(1x)(y1)\ar[r]^{a_{1x,y,1}}\ar@{.>}[d]^{b(1,x,y,1)} 
& 
((1x)y)1\ar[r]^-{a^{-1}_{1,x,y}\cdot id}
&
(1(xy))1\ar[d]^-{(id\cdot c_{x,y})\cdot id}
\\
yx 
& 
(1y)(x1)\ar[l]^-{l_y\cdot r_x} 
& 
((1y)x)z\ar[l]^-{a^{-1}_{1y,x,1}}
&
(1(yx))1\ar[l]^-{a_{1,y,x}\cdot id}
}
\] }

\noindent
whose commutativity follows again from the coherence theorem for symmetric monoidal categories. As for the associo-commutator of $(\DD_{sm})_{ac}$, it follows from (\ref{associocommutador})-(\ref{commutador_b}) that it is given by the long composite isomorphism shown in the diagram
{\small
\[
\xymatrix@C=3pc@R=3pc{
(xy)(zt)\ar[r]^-{r^{-1}_{xy}\cdot id}\ar[ddd]^{b(x,y,z,t)}
&
((xy)1)(zt)\ar[r]^-{b(xy,1,z,t)} \ar@{}@<-18ex>[r]^{(A)}
&
((xy)z)(1t)\ar[r]^-{(id\cdot l^{-1}_z)\cdot l_t}
&
((xy)(1z))t\ar[r]^-{b(x,y,1,z)\cdot id}\ar@{.>}[d]_-{(id\cdot(r^{-1}_zl_z)\cdot id}
&
((x1)(yz))t\ar[d]|-{(r_x\cdot (l_y^{-1}\cdot r_z^{-1}))\cdot id}
\\
&&&
((xy)(z1))t\ar@{.>}[d]_-{b(x,y,z,1)\cdot id}\ar@{}@<-7ex>[r]^{(B)}
& 
(x((1y)(z1)))t\ar[d]|-{(id\cdot b(1,y,z,1))\cdot id}
\\
&&& 
((xz)(y1))t\ar@{.>}[d]_-{(id\cdot(l^{-1}_yr_y))\cdot id}
&
(x((1z)(y1)))t\ar[d]|-{(r^{-1}_x\cdot (l_z\cdot r_y))\cdot id}
\\
(xz)(y) 
& 
((xz)1)(yt)\ar[l]^-{r_{xz}\cdot id} 
& 
((xz)y)(1t)\ar[l]^-{b(xz,y,1,t)} 
&  
((xz)(1y))t\ar[l]^-{(id\cdot l_y)\cdot l^{-1}_t} 
&
((x1)(zy))t\ar[l]^{b(x,1,z,y)\cdot id}
}
\] }

\noindent
Now, as indicated by the dotted arrows, this diagram can be subdivided into the subdiagrams (A) and (B). After cancelling some left and right unitors using the functoriality of the product, subdiagram (A) coincides with the commutative diagram (\ref{eq22}), and (B) with the product of the commutative diagram (\ref{eq21}) with $id_t$ (here we are using that $b(x,z,y,t)$ is the inverse of $b(x,y,z,t)$; Proposition~\ref{cond_simetria}). Therefore, both subdiagrams commute and the outer diagram is commutative. 
\end{proof}

As immediate consequences of Theorem~\ref{teorema_equivalencia}, we obtain both a coherence and a strictification theorem for AC-categories. 

\begin{corollary}\label{teorema_coherencia} {\rm (Coherence Theorem)}
Every diagram in an AC-category  built formally from associo-commutators and left and right unitors commutes.
\end{corollary}

\begin{proof}
Every such diagram in an AC-category $\CC$ will correspond to a diagram in $\CC_{sm}$ formally built from associators, commutators, and left and right unitors, and this diagram commutes by the coherence theorem for symmetric monoidal categories.
\end{proof}

\begin{corollary} {\rm (Strictification Theorem)}
Every AC-category is equivalent to a semistrict AC-category.
\end{corollary}

\begin{proof}
Let $\CC$ be an AC-category, and let $\CC^{s}_{sm}$ be any symmetric \emph{strict} monoidal category equivalent to $\CC_{sm}$, i.e., one whose associators and left and right unitors are identities. Since every 2-functor preserves equivalences we have
\[
\CC=(\CC_{sm})_{ac}\simeq(\CC^{s}_{sm})_{ac},
\]
and $(\CC^{s}_{sm})_{ac}$ is a semistrict AC-category. Indeed, it is clearly unital because $\CC^s_{sm}$ is unital. As for the associo-commutator, it is given by (\ref{associocommutador}) with $a,c$ the associator and commutator of $\CC^s_{sm}$. However, by hypothesis, the associator $a$ of $\CC^s_{sm}$ is the identity so that $b(x,1,z,t)$ reduces to the isomorphism $id_x\cdot c_{1,z}\cdot id_t$, and $c_{1,z}=id_z$ because in every symmetric monoidal category we have $c_{1,x}=l_x\,r^{-1}_x$ and $\CC^s_{sm}$ is unital. The fact that the isomorphisms $b(x,y,1,t)$ are also identities follows now from Proposition~\ref{cond_simetria}.
\end{proof}

It readily follows from (\ref{associador_b}) that the above isomorphism of 2-categories $\mathbf{ACCat}\cong\mathbf{SMCat}$ maps semistrict AC-categories to symmetric strict monoidal categories. Consequently, as is done below, brackets in such an AC-category can be omitted when writing the product of more than two objects.

Next result proves that only a few components of the associo-commutator in a semistrict AC-category need to be specified, giving precisely to the commutator in the associated symmetric strict monoidal category.

\begin{proposition}\label{AC-categories_semistrictes}
In every semistrict AC-category $\CC=(\Cc,\cdot,1,b)$ the associo-commutator is completely given by the components $b(1,x,y,1):xy\to yx$ for each $x,y$ in $\Cc$. Moreover, these isomorphisms are such that:
\begin{itemize}
\item[(sac1)] the diagram
\begin{equation}\label{cond_b(1,x,y,1)}
\xymatrix@C=3pc{
xyz\ar[dr]_-{b(1,x,y,1)\cdot id_z\ \ } \ar[rr]^-{b(1,x,yz,1)}
&& 
yzx
\\
&
yxz\ar[ur]_-{\ \ id_y\cdot b(1,x,z,1)}
&
}
\end{equation}
commutes for every objects $x,y,z$;
\item[(sac2)] $b(1,y,x,1)\,b(1,x,y,1)=id_{xy}$ for every objects $x,y$;
\item[(sac3)] $b(1,1,x,1)=id_x$ for every object $x$.
\end{itemize}
Conversely, for any triple $(\Cc,\cdot,1)$ as before, with $\cdot$ strictly associative and $1$ acting as a strict unit, and any family of natural isomorphisms $c_{x,y}:xy\to yx$ satisfying (sac1)-(sac3) with $c_{x,y}$ instead of $b(1,x,y,1)$ for every objects $x,y$, the isomorphisms
\begin{equation}\label{b(1,x,y,1)}
b(1,x,y,1):=c_{x,y}
\end{equation}
give the associo-commutator of a semistrict AC-category $(\Cc,\cdot,1,b)$.
\end{proposition}

\begin{proof}
Since $\CC$ is semistrict, we have
\[
b(x,y,z,t)\overset{(\ref{eq22})}{=}b(x,y,z,1)\cdot id_t\overset{(\ref{eq21})}{=}id_x\cdot b(1,y,z,1)\cdot id_t
\]
for every quadruple $(x,y,z,t)$. This proves the first assertion. Properties (sac2) and (sac3) follow readily from Propositions~\ref{cond_simetria} and \ref{lema0}, respectively, while the commutativity of (\ref{cond_b(1,x,y,1)}) follows from (acc1) with $x=y=t=z'=t'=1$, $z=x$, $x'=y$ and $y'=z$ together with (sac2). Indeed, for these objects axiom (acc1) in a semistrict AC-category reduces to the equality
\[
b(1,x,y,1)\cdot id_z=b(y,z,x,1)\,b(1,x,yz,1),
\]
and the right hand side is equal to $[id_y\cdot b(1,z,x,1)]\,b(1,x,yz,1)$ by (\ref{eq21}). The claim follows now from (sac2). To prove the converse, note that the isomorphisms $c_{x,y}$ are required to satisfy the following axioms:
\begin{itemize}
\item[(ssm1)] the diagram
\begin{equation}\label{cond_c}
\xymatrix@C=3pc{
xyz\ar[dr]_-{c_{x,y}\cdot id_z} \ar[rr]^-{c_{x,yz}}
&& 
yzx
\\
&
yxz\ar[ur]_-{id_y\cdot c_{x,z}}
&
}
\end{equation}
commutes for every objects $x,y$;
\item[(ssm2)] $c_{y,x}\,c_{x,y}=id_{xy}$ for every objects $x,y$;
\item[(ssm3)] $c_{1,x}=id_x$ for every object $x$.
\end{itemize}
Now, such a quadruple $\mathbb C=(\mathcal C,\cdot,1,c)$, together with the trivial associator and trivial left and right unitors, is precisely the data defining a symmetric \emph{strict} monoidal category, and the above quadruple $(\mathcal C,\cdot,1,b)$ with $b$ defined by (\ref{b(1,x,y,1)}) is nothing but the associated (semistrict) AC-category $\mathbb C_{ac}$ by the isomorphism in Theorem~\ref{teorema_equivalencia}.
\end{proof}

\begin{remark}\label{remarca_final}
As pointed out in the Introduction, AC-categories do not provide an alternative description of arbitrary braided monoidal categories (see Proposition~\ref{cond_simetria}). Such a description is instead provided by the unital $b$-categories of Davydov and Runkel. Basically, a unital $b$-category is a category $\mathcal C$ equipped with a product functor $\cdot:\mathcal C\times\mathcal C\to\mathcal C$, a distinguished object $1$, and natural isomorphisms
\begin{align*}
\beta_{x,y,z}&:x(yz)\to y(xz),
\\ 
\rho_x&:x1\to x,
\end{align*}
satisfying appropriate coherence conditions; see \cite{Davydov-Runkel-2015} for details. In Section~2.5 of their work, Davydov and Runkel explain how a braided monoidal category can be obtained from a unital $b$-category. Furthermore, a unital $b$-category can also be constructed from an AC-category. Indeed, if $\mathbb C=(\mathcal C,\cdot,1,b,l,r)$ is an AC-category, and we define $\beta_{x,y,z}$ as the composite morphism
$$
x(yz)\overset{l_x^{-1}\cdot \mathrm{id}}{\longrightarrow} (1x)(yz)\overset{b(1,x,y,z)}{\longrightarrow}(1y)(xz)\overset{l_y\cdot \mathrm{id}}{\longrightarrow} y(xz),
$$
and $\rho_x:x1\to x$ as the right unitor $r_x$ of $\mathbb C$, then $\mathbb C_b=(\mathcal C,\cdot,1,\beta,\rho)$ is a unital $b$-category. The coherence axioms required for $\beta$ and $\rho$ hold true because they reduce to diagrams involving the associo-commutator and the left and right unitors of $\mathbb C$, which commute by Corollary~\ref{teorema_coherencia}. It turns out that the braided monoidal category associated with $\mathbb C_b$ is, in fact, the symmetric monoidal category $\mathbb C_{\mathrm{sm}}$ from Theorem~\ref{teorema_equivalencia}. Indeed, one can readily verify that the braiding and left unitor defined by Davydov and Runkel from $\beta$ and $\rho$ coincide with the canonical commutator given by (\ref{commutador_b}) and the original left unitor of $\mathbb C$, while their associator necessarily equals the canonical associator given by (\ref{associador_b}) by virtue of Corollary~\ref{teorema_coherencia}.
\end{remark}

\section{Cubical cohomology of abelian groups}

\subsection{Eilenberg-MacLane cubical complex}
In 1950, Eilenberg and MacLane \cite{Eilenberg-MacLane-1950-I} showed how to assign functorially to an abelian group $\Asf$ a non-negative chain complex $Q_\bullet(\Asf)$ of abelian groups, called the {\em $Q$-construction} or ({\em normalized}) {\em cubical complex}, whose homology group $H_n(Q_\bullet(\Asf))$ for $n\geq 0$ is isomorphic to $H_{n+k}(K(\Asf,k))$ for each $k\geq n+1$, where $K(\Asf,k)$ denotes the Eilenberg-MacLane space with all homotopy groups trivial except $\pi_k$, which is equal to $\Asf$. The complex $Q_\bullet(\Asf)$ is obtained as a suitable normalization of an auxiliary {\em unnormalized} cubical complex $Q'_\bullet(\Asf)$. 

By definition, $Q'_n(\Asf)$ for each $n\geq 0$ is the free abelian group generated by the $n$-dimensional {\em $A$-cubes}, i.e. $n$-dimensional cubes whose vertices are labelled by elements in $A$. More formally, let $C_n$ be the set of all binary words of length $n$, with $C_0$ consisting of just the empty word $(\,)$. The elements in $C_n$ can indeed be identified with the vertices of a $n$-dimensional cube. Then the free generators of $Q'_n(\Asf)$ are the maps $X:C_n\to A$. We shall write
\[
X(\epsilon_1,\ldots,\epsilon_n)=a_{\epsilon_1\cdots\epsilon_n}\in A
\]
for each $(\epsilon_1,\ldots,\epsilon_n)\in C_n$. For $n=0,1,2,3$ the free generators of $Q'_n(\Asf)$, or $n$-generators for short, are respectively denoted by
\[
(a),\quad (a_0,a_1),\quad \begin{pmatrix} a_{00}&a_{01} \\ a_{10}&a_{11}\end{pmatrix},\quad \left(\begin{array}{cc|cc}a_{000}&a_{001}&a_{100}&a_{101}\\a_{010}&a_{011}&a_{110}&a_{111}\end{array}\right).
\]
Notice that for each $n\geq 1$ an $n$-generator is just a pair of $(n-1)$-generators, representing two opposite faces of the $n$-dimensional $A$-cube. This is made explicit in the 3-generator above, where the $2\times 2$ submatrices on either side of the vertical line represent opposite faces of the three-dimensional $A$-cube. In order to define the differential of $Q'_\bullet(A)$ it is convenient to introduce the maps $0_i,1_i:C_n\to C_{n+1}$ for each $i=1,\ldots,n+1$ given by
\begin{align*}
0_i(\epsilon_1,\ldots,\epsilon_n)&:=(\epsilon_1,\ldots,\epsilon_{i-1},0,\epsilon_i,\ldots,\epsilon_n),
\\
1_i(\epsilon_1,\ldots,\epsilon_n)&:=(\epsilon_1,\ldots,\epsilon_{i-1},1,\epsilon_i,\ldots,\epsilon_n)
\end{align*}
whose images are called the lower and upper $i$-faces of $C_{n+1}$, respectively. They correspond to two opposite $n$-dimensional faces of $C_{n+1}$. Then the differential $\delta_n:Q'_n(A)\to Q'_{n-1}(A)$ for each $n\geq 1$ is the group homomorphism defined on the free generators by
\[
\delta_n X:=\sum_{i=1}^n(-1)^i(S^{(i)}_{n}(X)-U^{(i)}_{n}(X)-L^{(i)}_{n}(X)),
\]
with $U^{(i)}_{n},L^{(i)}_{n}:Q'_n(A)\to Q'_{n-1}(A)$ the restrictions to the $ith$ upper and lower faces, respectively, i.e.
\begin{align*}
U^{(i)}_{n}(X)(\epsilon_1,\ldots,\epsilon_{n-1})&:=X(1_i(\epsilon_1,\ldots,\epsilon_{n-1})), 
\\
L^{(i)}_{n}(X)(\epsilon_1,\ldots,\epsilon_{n-1})&:=X(0_i(\epsilon_1,\ldots,\epsilon_{n-1})),
\end{align*}
and $S^{(i)}_{n}:Q'_n(A)\to Q'_{n-1}(A)$ the vertexwise addition of $U^{(i)}_{n}$ and $L^{(i)}_{n}$. For instance, on the free $n$-generators with $1\leq n\leq 3$ the differential is given by
\begin{equation}\label{delta1}
\delta_1\,(x,y)=(x)+(y)-(x+y)
\end{equation}
\begin{equation}\label{delta2}
\delta_2\,\begin{pmatrix}x&y \\ z&t\end{pmatrix}=(x,y)+(z,t)-(x+z,y+t)-(x,z)-(y,t)+(x+y,z+t)
\end{equation}
\begin{align}
\delta_3\left(\begin{array}{cc|cc}x&y&x'&y'\\z&t&z'&t'\end{array}\right)&=\begin{pmatrix}x&y\\ z&t\end{pmatrix}+\begin{pmatrix}x'&y'\\ z'&t'\end{pmatrix}-\begin{pmatrix}x+x'&y+y'\\ z+z'&t+t'\end{pmatrix} \label{delta3}
\\
&\hspace{1.5truecm} +\begin{pmatrix}x+z&y+t\\ x'+z'&y'+t'\end{pmatrix}-\begin{pmatrix}x&y\\ x'&y'\end{pmatrix}-\begin{pmatrix}z&t\\ z'&t'\end{pmatrix} \nonumber
\\
&\hspace{3truecm}+\begin{pmatrix}x&z\\ x'&z'\end{pmatrix}+\begin{pmatrix}y&t\\ y'&t'\end{pmatrix}-\begin{pmatrix}x+y&z+t\\ x'+y'&z'+t'\end{pmatrix} \nonumber
\end{align}
By definition, the (normalized) cubical complex $Q_\bullet(\Asf)$ is the quotient of $Q'_\bullet(\Asf)$ modulo the subcomplex $N_\bullet(\Asf)$ generated by the so called {\em slabs} and {\em diagonals}. A free $n$-generator $X$ is called a {\em slab} of $Q'_n(A)$, or a {\em $n$-slab}, if $X=(0)$ in case $n=0$, or if his restriction to one of its faces is identically zero (i.e. all the vertices in the face are labelled by $0\in A$) in case $n\geq 1$. More  precisely, for each $n\geq 1$ and each $i=1,\ldots,n$ an $i$-{\em slab} of $Q'_n(A)$ is a free $n$-generator $X$ such that either $X(0_ie)=0$ for each $e\in C_{n-1}$, or $X(1_ie)=0$ for each $e\in C_{n-1}$. For instance, for $1\leq n\leq 3$ the slabs of $Q'_n(\Asf)$ are
\begin{eqnarray*}
(0,y),\,(x,0),\,
\begin{pmatrix}0&0\\x&y\end{pmatrix},\,\begin{pmatrix}x&y\\0&0\end{pmatrix},\,\begin{pmatrix}0&x\\0&y\end{pmatrix},\,\begin{pmatrix}x&0\\y&0\end{pmatrix},\,
\\
\ \ \left(\begin{array}{cc|cc}0&0&x&y\\0&0&z&t\end{array}\right), \left(\begin{array}{cc|cc}x&y&0&0\\z&t&0&0\end{array}\right),\,\left(\begin{array}{cc|cc}0&0&0&0\\x&y&z&t\end{array}\right),\,
\\
\left(\begin{array}{cc|cc}x&y&z&t\\0&0&0&0\end{array}\right),\,
\left(\begin{array}{cc|cc}0&x&0&y\\0&z&0&t\end{array}\right), \left(\begin{array}{cc|cc}x&0&y&0\\z&0&t&0\end{array}\right).\  
\end{eqnarray*}
for each $x,y,z,t\in A$. A free $n$-generator $X$ for $n\geq 2$ is called a {\em $n$-diagonal} if $X(\epsilon_1,\ldots,\epsilon_n)$ is zero for all $(\epsilon_1,\ldots,\epsilon_n)$ with $\epsilon_i\neq \epsilon_{i+1}$ for some $i\in\{1,\ldots,n-1\}$ (in this case, $X$ is called an {\em i-diagonal}). For instance, for $2\leq n\leq 3$ the diagonals are the generators
\[
\begin{pmatrix}x&0\\ 0&y\end{pmatrix},\, \left(\begin{array}{cc|cc}x&y&0&0\\0&0&z&t\end{array}\right),\,\left(\begin{array}{cc|cc}x&0&z&0\\0&y&0&t\end{array}\right)
\]
for each $x,y,z,t\in A$. Then $N_n(\Asf)$ is the subgroup of $Q'_n(\Asf)$ generated by the $n$-slabs and $n$-diagonals. It is easy to check that $N_\bullet(\Asf)$ is indeed a subcomplex of $Q'_\bullet(\Asf)$, and
\[
Q_\bullet(\Asf):=Q'_\bullet(\Asf)/N_\bullet(\Asf).
\]
Notice that when $\Asf=0$ there is only one $n$-dimensional $A$-cube for each $n\geq 0$ so that $Q'_n(0)\cong\ZZ$ for each $n\geq 0$. Moreover, all the maps $U^{(i)}_n,L^{(i)}_n,S^{(i)}_n:\ZZ\to\ZZ$ are identities. Hence the differential $\delta_n:\ZZ\to\ZZ$ is
\[
\delta_n=\left\{\begin{array}{cl}
id_\ZZ, & \mbox{if $n$ is odd} 
\\
0, & \mbox{if $n$ is even.}
\end{array}\right.
\]
In particular, $Q'_\bullet(0)$ has non-trivial homology. However, we have $N_n(0)=Q'_n(0)$ for each $n\geq 0$ so that $Q_\bullet(0)$ is indeed the zero chain complex.

\subsection{Cubical cohomology}

By definition, the cubical cohomology of an abelian group $\Asf$ with coefficients in another abelian group $\Bsf$ is the homology of the cochain complex
\[
C_{cub}^{\bullet}(\Asf,\Bsf):=Hom(Q_{\bullet-1}(\Asf),\Bsf).
\]
i.e. of the complex
\[
0\to Hom(Q_0(\Asf),\Bsf)\overset{\partial^1}{\to} Hom(Q_1(\Asf),\Bsf)\to \cdots\to Hom(Q_{n-1}(\Asf),\Bsf)\overset{\partial^n}{\to} Hom(Q_n(\Asf),\Bsf)\to \cdots
\]
with the homomorphisms $f:Q_n(\Asf)\to\Bsf$ as cubical $(n+1)$-cochains, and $\partial^n=-\circ\delta_n$ for each $n\geq 1$. Its homology will be denoted by $\Hsf_{cub}^\bullet(\Asf,\Bsf)$. At first degree we have
\[
\Hsf^1_{cub}(\Asf,\Bsf)=\mathsf{Hom}(\Asf,\Bsf)
\]
so that $\Hsf^1_{cub}(\Asf,\Bsf)$ coincides with the usual degree-one group cohomology $\Hsf^1(\Asf,\Bsf)$ when $\Bsf$ is viewed as a trivial $\Asf$-module. However, the two cohomologies differ in degree $n\geq 2$. In fact, unlike the usual group cohomology $\Hsf^\bullet(\Asf,\Bsf)$, the cubical one takes explicitly account of the abelian character of $\Asf$ because the condition $\delta\circ\delta=0$ requires both the associativity and the commutativity of the sum in $\Asf$ to be true (in the usual complex only the associativity is required). Moreover, while $\Hsf^2(\Asf,\Bsf)$ classifies the central extensions of $\Bsf$ by $\Asf$, $\Hsf_{cub}^2(\Asf,\Bsf)$ classifies the extensions that are abelian.

We are interested in the third-degree cubical cohomology because, as shown in the next section, it provides an alternative cohomological classification of symmetric 2-groups. It readily follows from (\ref{delta3}) that a cubical 3-cocycle on $\Asf$ with coefficients in $\Bsf$ amounts to a map $\zsf:\Asf^4\to\Bsf$ such that
\begin{align}\label{cond_cocicle}
\zsf(x,y,z,t)+&\zsf(x',y',z',t')-\zsf(x+x',y+y',z+z',t+t') \nonumber
\\
&+\zsf(x,z,x',z')+\zsf(y,t,y',t')-\zsf(x+y,z+t,x'+y',z'+t') \nonumber
\\
&\hspace{1truecm}=\zsf(x,y,x',y')+\zsf(z,t,z',t')-\zsf(x+z,y+t,x'+z',y'+t')
\end{align}
for every $x,y,z,t,x',y',z',t'\in A$, and satisfying the normalization conditions
\begin{equation}\label{normalitzacio_cocicle}
\zsf(x,y,0,0)=\zsf(0,0,x,y)=\zsf(x,0,0,y)=\zsf(0,x,0,y)=\zsf(x,0,y,0)=0
\end{equation}
for every $x,y\in A$. Moreover, it follows from (\ref{delta2}) that such a cubical 3-cocycle is a cubical 3-coboundary if there exists a map $\csf:A^2\to B$, with
\[
\csf(0,x)=\csf(x,0)=0
\]
for each $x\in A$, and such that
\begin{equation}\label{cond_covora}
\zsf(x,y,z,t)=\csf(x,y)+\csf(z,t)-\csf(x+z,y+t)-\csf(x,z)-\csf(y,t)+\csf(x+y,z+t)
\end{equation}
for every $x,y,z,t\in A$. Thus every cubical 3-coboundary is antisymmetric in the middle two arguments. In fact, as we shall see in the next section, every cubical 3-cocycle is antisymmetric in the middle two arguments (see Corollary~\ref{propietat_antisimetria_z}).

\section{A new cohomological classification of symmetric 2-groups} 

Since symmetric 2-groups are a particular type of symmetric monoidal category, it follows from Theorem 2.9 that they can also be regarded as a particular type of AC-category. Let us make this precise, and for convenience we adopt the additive notation $(+,0)$ instead of $(\cdot,1)$.

\begin{definition}
An {\em AC-2-group} is an AC-category $\CC=(\Cc,+,0,b,l,r)$ whose underlying category $\Cc$ is a groupoid and such that every object has a (weak) inverse with respect to $+$ (i.e. an object $-x$ such that $x+(-x),(-x)+x\cong 0$).
\end{definition}

\begin{proposition}\label{equivalencia_AC2grup_2grupsimetric}
Let $\mathbf{2ACGrp}$ be the full sub-2-category of $\mathbf{ACCat}$ with objects the AC-2-groups, and  $\mathbf{2SGrp}$  the full sub-2-category of $\mathbf{SMCat}$ with objects the symmetric 2-groups. Then there is a canonical isomorphism of 2-categories $\mathbf{2SGrp}\cong\mathbf{2ACGrp}$.
\end{proposition}

\begin{proof}
When $\CC$ is a symmetric 2-group the corresponding AC-category $\CC_{ac}$ is an AC-2-group, and when $\DD$ is an AC-2-group the corresponding symmetric monoidal category $\DD_{sm}$ is a symmetric 2-group. Hence the isomorphism $\mathbf{SMCat}\cong\mathbf{ACCat}$ of Theorem~\ref{teorema_equivalencia} indeed restricts to an isomorphism between $\mathbf{2SGrp}$ and $\mathbf{2ACGrp}$.
\end{proof}

Thus, AC-2-groups are equivalent to symmetric 2-groups. Consequently, a new cohomological classification of symmetric 2-groups can be obtained by classifying AC-2-groups up to equivalence. As we now show, the relevant cohomology is Eilenberg–MacLane cubical cohomology for abelian groups.

\subsection{}
To classify AC-2-groups up to equivalence, we adopt an approach analogous to that of Sinh and of Joyal–Street in their respective classifications of symmetric and braided 2-groups. We begin by showing that a distinguished class of skeletal AC-2-groups, which we call special AC-2-groups, can be constructed from any triple $(\Gsf,\Asf,\zsf)$ consisting of two abelian groups $\Gsf,\Asf$, and a cubical 3-cocycle $\zsf$ on $\Gsf$ with values in $\Asf$ (\S~\ref{AC2-grup_especial}). Furthermore, cohomologous cubical 3-cocycles yield equivalent AC-2-groups (Proposition~\ref{equivalents_quan_z_sim_z'}). We then show that every AC-2-group is equivalent to a special AC-2-group by explicitly determining the corresponding triple $(\Gsf,\Asf,\zsf)$ (Proposition~\ref{terna}).

\subsection{}\label{AC2-grup_especial}

Let $(\Gsf,\Asf,\zsf)$ be any triple consisting of two abelian groups $\Gsf,\Asf$ together with a cubical 3-cocycle of $\Gsf$ with values in $\Asf$. Then an AC-2-group is given as follows. The underlying groupoid has the elements of $G$ as objects, the elements of $A\times G$ as morphisms, with $(a,x):x\to x$, and the morphisms of the form $(0,x)$ as identity morphisms. The composition law is given by
\begin{equation}\label{composicio}
(a,x)\circ(b,x)=(a+b,x),
\end{equation}
while the sum is given on objects by the product in $G$, and on morphisms by
\begin{equation}\label{suma}
(a,x)+(b,y)=(a+b,x+y).
\end{equation}
The identity $0\in G$ is the zero object, and the left and right unitors are both identities. Finally, the associo-commutator is given by
\begin{equation}\label{associo-commutador_especial}
b(x,y,z,t)=(\zsf(x,y,z,t),x+y+z+t)
\end{equation}
for every $x,y,z,t\in G$. Axioms (acc1)-(acc3) on $b$ follow readily from (\ref{composicio}), (\ref{suma}) and conditions (\ref{cond_cocicle}) and (\ref{normalitzacio_cocicle}) on $\zsf$. For instance, an easy computation using (\ref{composicio}) and (\ref{suma}) shows that axiom (acc1) amounts to the equality 
\begin{align}
\zsf(x+y,z+t,x'+y',z'+t')&+\zsf(x,y,x',y')+\zsf(z,t,z',t')+\zsf(x+x',y+y',z+z',t+t') \nonumber
\\
&=\zsf(x,y,z,t)+\zsf(x',y',z',t')+\zsf(x+z,y+t,x'+z',y'+t')\nonumber
\\
&\hspace{2truecm}+\zsf(x,z,x',z')+\zsf(y,t,y',t')
\end{align}
which, after rearranging terms, reduces to (\ref{cond_cocicle}). Axioms (acc2)-(acc3) follow similarly from (\ref{normalitzacio_cocicle}). Hence the structure so defined is indeed a (skeletal) AC-2-group. It will be denoted by $\AAA(\Gsf,\Asf,\zsf)$. 

\begin{definition}
An AC-2-group $\AAA$ is called a {\em special AC-2-group} if $\AAA=\AAA(\Gsf,\Asf,\zsf)$ for some triple $(\Gsf,\Asf,\zsf)$ as before.
\end{definition}

Notice that the relation between the associo-commutator $b$ in a special AC-2-group and the corresponding cubical 3-cocycle $\zsf$, as given by (\ref{associo-commutador_especial}), implies that every property of $b$ can be translated into an equivalent property of $\zsf$. In particular, we have the following.

\begin{corollary}\label{propietat_antisimetria_z}
Let $\Gsf,\Asf$ be any abelian groups. Then every (normalized) cubical 3-cocycle $\zsf:\Gsf^4\to\Asf$ is antisymmetric in the middle two arguments.
\end{corollary}

\begin{proof}
Let us consider the special AC-2-group $\AAA(\Gsf,\Asf,\zsf)$. We know from Proposition~\ref{cond_simetria} that its associo-commutator is such that the composite isomorphism $b(x,y,z,t)\,b(x,z,y,t)$ is the identity for every objects $x,y,z,t$, and in terms of the cubical 3-cocycle $\zsf$ this means that $\zsf(x,y,z,t)+\zsf(x,z,y,t)=0$. 
\end{proof}

\begin{proposition}\label{equivalents_quan_z_sim_z'}
Let $(\Gsf,\Asf,\zsf)$ be as before, and let $\csf$ be any cubical 2-cochain of $\Gsf$ with values in $\Asf$. Then $\AAA(\Gsf,\Asf,\zsf+\partial\csf)\simeq\AAA(\Gsf,\Asf,\zsf)$.
\end{proposition}

\begin{proof}
Let $\FF:\AAA(\Gsf,\Asf,\zsf)\to\AAA(\Gsf,\Asf,\zsf+\partial\csf)$ be the AC-functor given as follows. The underlying functor is the identity, the natural isomorphisms $F_2(x,y):x+y\to x+y$ are given by
\[
F_2(x,y):=(\csf(x,y),x+y)
\]
and $F_1$ is the identity. Axiom (acf1) readily follows from the fact that the cubical 3-cocycle of the target AC-2-group is $\zsf'=\zsf+\partial\csf$, and axiom (acf2) follows from the normalization conditions on $\csf$.  Therefore $\FF$ is indeed an AC-functor, and it is clearly an equivalence of AC-2-groups.
\end{proof}

\begin{proposition}\label{terna}
Every AC-2-group is equivalent to a special AC-2-group $\AAA(\Gsf,\Asf,\zsf)$ for some triple $(\Gsf,\Asf,\zsf)$ as before.
\end{proposition}
 
\begin{proof}
Let $\AAA=(\Aa,\cdot,+,0,l,r)$ be any AC-2-group. The corresponding abelian groups $\Gsf,\Asf$ are respectively given by the the group $\pi_0(\AAA)$ of isomorphism classes of objects in $\Aa$ with the composition law induced by $+$, and the group $\pi_1(\AAA)$ of automorphisms of the unit object $0$. Both are abelian, the first because $x+y\cong y+x$ for every objects $x,y$ in $\Aa$, and the second by the Eckmann-Hilton argument applied to the composition of automorphisms and the commutative operation $\hat +$ defined by $a\hat + b=d\,(a+b)\,d^{-1}$ for each $a,b\in A$, which proves that $\circ=\hat +$ (as before, $d$ denotes the canonical isomorphism $0+0\cong 0$). 
To obtain the cubical 3-cocycle $\zsf$ we have to choose a representative $x_g$ for each $g\in\pi_0(\AAA)$, with $x_0=0$, and for each object $x$ an isomorphism $\iota_x:x\to x_g$ if $[x]=g$, with $\iota_{x_g}$ the identity for every $g$. Then let $\Aa_{sk}$ be the skeletal groupoid whose objects are the elements $g\in\pi_0(\AAA)$, whose morphisms are the elements $(u,g)\in \pi_1(\AAA)\times\pi_0(\AAA)$, with $(u,g):g\to g$, and whose composition law is given by $\hat +$. Let us denote by $H:\Aa_{sk}\to\Aa$ the functor given on objects and morphisms by 
\begin{align}
H(g)&:=x_{g},\label{def_H_objectes}
\\
H(u,g)&:=\delta_{x_g}(u).\label{def_H_morfismes}
\end{align}
Here $\delta_{x_g}:\mathsf{Aut}(0)\to \mathsf{Aut}(x_g)$ stands for the canonical group isomorphism given by
\[
\delta_{x_g}(u)=r_{x_g}\,(id_{x_g}+u)\,r^{-1}_{x_g},
\]
with $r_{x_g}:x_g+0\to x_g$ the right unitor. Its existence follows from Proposition~\ref{equivalencia_AC2grup_2grupsimetric}, which establishes that the underlying groupoid of an AC-2-group is of the same type as the underlying groupoid of a symmetric 2-group. Then $H$ is an equivalence of categories with pseudoinverse $H':\Aa\to\Aa_{sk}$ given by
\begin{align}
H'(x)&:=[x],\label{def_H'_objectes}
\\
H'(f)&:=(\delta^{-1}_{x_g}(\iota_{x'}\,f\,\iota_x^{-1}),[x]).\label{def_H'_morfismes}
\end{align}
for every morphism $f:x\to x'$ in $\Aa$ (hence, $[x]=[x']$). In fact, $H'\circ H$ is literally equal to the identity of $\Aa_{sk}$, and the isomorphisms $\iota_x$ are the components of a natural isomorphism $\iota:id_\Aa\Rightarrow H\circ H'$. By transport of the AC-2-group structure on $\Aa$ to $\Aa_{sk}$ through the equivalence $H'$ the skeletal groupoid $\Aa_{sk}$ becomes a skeletal AC-2-group $\AAA_{sk}$ equivalent to $\AAA$. The resulting AC-2-group $\AAA_{sk}$ turns out to be the special AC-2-group $\AAA(\Gsf,\Asf,\zsf)$ with $\zsf:\Gsf\times\Gsf\to\Asf$ the cubical 3-cocycle uniquely given by the commutative diagrams
\begin{equation}\label{3-cocicle_z}
\xymatrix@C=3pc@R=3pc{
(x_{g_1}+x_{g_2})+(x_{g_3}+x_{g_4})\ar[rr]^-{\iota_{x_{g_1}+x_{g_2}}+\iota_{x_{g_3}+x_{g_4}}}\ar[d]|{b(x_{g_1},x_{g_2},x_{g_3},x_{g_4})} 
&& 
x_{g_1+g_2}+x_{g_3+g_4}\ar[rr]^-{\iota_{x_{g_1+g_2}+x_{g_3+g_4}}} 
&& 
x_{g_1+g_2+g_3+g_4}\ar[d]|{H(\zsf(g_1,g_2,g_3,g_4),g_1+g_2+g_3+g_4)} 
\\
(x_{g_1}+x_{g_3})+(x_{g_2}+x_{g_4})\ar[rr]_-{\iota_{x_{g_1}+x_{g_3}}+\iota_{x_{g_2}+x_{g_4}}}
&& 
x_{g_1+g_3}+x_{g_2+g_4}\ar[rr]_-{\iota_{x_{g_1+g_3}+x_{g_2+g_4}}} 
&& 
x_{g_1+g_3+g_2+g_4}
}
\end{equation}
for every $g_1,g_2,g_3,g_4\in\pi_0(\AAA)$. Notice that, as a consequence of the possible non-skeletal character of the original AC-2-group, we may have $\zsf(g_1,g_2,g_3,g_4)\neq 0$ even if the associo-commutator $b(x_{g_1},x_{g_2},x_{g_3},x_{g_4})$ is the identity. The cubical 3-cocycle condition on $\zsf$ follows from the $4\times 4$ axiom on $b$, while the normalization conditions follow from the unital and normalization axioms. Clearly, the 3-cocycle $\zsf$ depends on the choices of the representative objects $x_g$ and the isomorphisms $\iota_x$. However, it can be shown that making a different choice just leads to a cohomologous 3-cocycle. The details are left to the reader.
\end{proof}

\begin{theorem}\label{classificacio_AC-2-grups}
Up to equivalence, an AC-2-group (and hence, a symmetric 2-group) is completely determined by a third-degree cubical cohomology class $[\zsf]\in\Hsf_{cub}^3(\Gsf,\Asf)$ of some abelian group $\Gsf$ with coefficients in some other abelian group $\Asf$.
\end{theorem}

\begin{proof}
It is a consequence of Propositions~\ref{terna} and \ref{equivalents_quan_z_sim_z'}.
\end{proof}

In fact, different cohomology classes $[\zsf],[\zsf']\in\Hsf_{cub}^3(\Gsf,\Asf)$ may correspond to equivalent AC-2-groups. More generally, the special AC-2-groups $\GG(\Gsf,\Asf,\zsf)$ and $\GG(\Gsf',\Asf',\zsf')$ are equivalent if and only if there exist group isomorphisms $g:\Gsf\to\Gsf'$ and $f:\Asf\to\Asf'$ such that $g^*(\zsf')\sim f_*(\zsf)$ in $\Hsf^3(\Gsf,\Asf')$. The situation is completely analogous to the classification of arbitrary 2-groups in terms of the usual group cohomology classes (see \cite{Baez-Lauda-2004}). The details are left to the reader. 

\subsection{}\label{comentari_final}
According to Theorem~\ref{classificacio_AC-2-grups}, AC-2-groups and hence, symmetric 2-groups can be classified by cubical cohomology classes in degree three. However, the usual classification of symmetric 2-groups, due to Sinh \cite{Sinh-1975}, is given in terms of  Eilenberg-MacLane level two abelian group cohomology in degree three, whose 3-cocycles consist of pairs $(\hsf,\csf)$ with $\hsf:G^3\to A$ a usual (normalized) 3-cocycle of $\Gsf$ with values in $\Asf$ as a trivial $\Gsf$-module, and $\csf:G^2\to A$ a skew-symmetric map such that
\begin{align*}
\csf(g_2,g_3)-\csf(g_1+g_2,g_3)+\csf(g_1,g_3)&=-\hsf(g_1,g_2,g_3)+\hsf(g_1,g_3,g_2)-\hsf(g_3,g_1,g_2),
\\
\csf(g_1,g_3)-\csf(g_1,g_2+g_3)+\csf(g_1,g_2)&=\hsf(g_1,g_2,g_3)-\hsf(g_2,g_1,g_3)+\hsf(g_2,g_3,g_1).
\end{align*}
In fact, Eilenberg and MacLane proved \cite{Eilenberg-MacLane-1951} that the cubical complex and the complex from which Sinh's 3-cocycles arise are homologically equivalent, though without providing an explicit equivalence. The present work can be interpreted as a categorical realization of this equivalence in degree three.

\noindent
\bibliography{bibliografia}{}

\begin{thebibliography}{10}

\bibitem{Aguiar-Mahajan-2010}
Marcelo Aguiar and Swapneel Mahajan.
\newblock Monoidal functors, species and {H}opf algebras.
\newblock {\em CRM Monograph Series}, 2010.

\bibitem{Baez-Lauda-2004}
J.~Baez and A.~Lauda.
\newblock Higher dimensional algebra {V}: 2-groups.
\newblock {\em Theory Appl. Categ.}, 12(14):423--491, 2004.

\bibitem{Davydov-Runkel-2015}
Alexei Davydov and Ingo Runkel.
\newblock An alternative description of braided monoidal categories.
\newblock {\em Applied Categorical Structures}, 23(3):279--309, 2015.

\bibitem{Dorsen-Petric-2007}
Kosta Do{\v{s}}en and Zoran Petri{\'c}.
\newblock Medial commutativity.
\newblock {\em Annals of Pure and Applied Logic}, 146(2-3):237--255, 2007.

\bibitem{Eilenberg-MacLane-1950-I}
S.~Eilenberg and S.~MacLane.
\newblock Cohomology theory of abelian groups and homotopy theory {I}.
\newblock {\em Proc. Nat. Acad. Sci. USA}, 36:443--447, 1950.

\bibitem{Eilenberg-MacLane-1951}
Samuel Eilenberg and Saunders MacLane.
\newblock Homology theories for multiplicative systems.
\newblock {\em Transactions of the American Mathematical Society},
  71(2):294--330, 1951.

\bibitem{Gurski-Johnson-2025}
Nick Gurski and Niles Johnson.
\newblock Universal pseudomorphisms, with applications to diagrammatic
  coherence for braided and symmetric monoidal functors.
\newblock {\em Compositionality}, 7, 2025.

\bibitem{Jibladze-Pirashvili-2007}
M.~Jibladze and T.~Pirashvili.
\newblock Third {M}ac{L}ane cohomology via categorical rings.
\newblock {\em J. Homotopy and Related Struc.}, 2(2):187--216, 2007.

\bibitem{Joyal-Street-1993}
A.~Joyal and R.~Street.
\newblock Braided monoidal categories.
\newblock {\em Adv. Math.}, 102:20--78, 1993.

\bibitem{MacLane-1958b}
Saunders Mac~Lane.
\newblock Extensions and obstructions for rings.
\newblock {\em Illinois Journal of Mathematics}, 2(3):316--345, 1958.

\bibitem{MacLane-1998}
Saunders Mac~Lane.
\newblock {\em Categories for the working mathematician}, volume~5.
\newblock Springer Science \& Business Media, 1998.

\bibitem{MacLane-1963}
S.~MacLane.
\newblock Natural associativity and commutativity.
\newblock {\em Rice Univ. Studies}, 49:28--46, 1963.

\bibitem{Quang-1987}
N.T. Quang.
\newblock Introduction to {A}nn-categories.
\newblock {\em J. Math. Hanoi}, 15(4):14--24, 1987.

\bibitem{Quang-2013}
N.T. Quang.
\newblock Cohomological classification of {A}nn-categories.
\newblock {\em Math. Communications}, 18:151--169, 2013.

\bibitem{Sinh-1975}
H.X. Sinh.
\newblock Gr-cat\'egories.
\newblock {\em These de Doctorat d'Etat, Univ. Paris VII}, 1975.

\bibitem{Yau-2024-I}
Donald Yau.
\newblock {\em Bimonoidal Categories, $ E\_n $-Monoidal Categories, and
  Algebraic $ K $-Theory: Volume I: Symmetric Bimonoidal Categories and
  Monoidal Bicategories}, volume 283.
\newblock American Mathematical Society, 2024.

\end{thebibliography}
\bibliographystyle{plain}

\end{document}